\documentclass{article}

\usepackage[utf8]{inputenc}
\usepackage[english]{babel}
\usepackage{amsmath}
\usepackage{color}
\usepackage{amsfonts,enumitem}
\usepackage{amssymb}
\usepackage{hyperref} 
\usepackage{graphicx,tikz}
\usepackage{cleveref}
\usepackage{bm}
\usepackage{bbm}
\usepackage{enumitem}
\usepackage{geometry}

\usepackage{algorithm} 
\usepackage{algorithmic} 
\usepackage{lipsum}
\usepackage[skins]{tcolorbox}

\newcommand*{\SET}[1]  {\ensuremath{\mathbb{#1}}}

\newcommand{\R}{\SET{R}}
\newcommand{\Z}{\SET{Z}}
\newcommand{\J}{\mathcal{J}}

\newcommand{\E}{\SET{E}}
\newcommand{\Q}{\SET{Q}}
\newcommand{\N}{\SET{N}}

\newcommand{\conv}{\operatorname{conv}}
\newcommand{\Graph}{\operatorname{Graph}}
\newcommand{\crit}{\operatorname{crit}}
\newcommand{\Jac}{\operatorname{Jac}}
\newcommand{\backprop}{\operatorname{backprop}}

\DeclareMathOperator{\di}{d\!}
\newcommand{\Ran}{\R_{\operatorname{an}}}
\newcommand{\Ranexp}{\R_{\operatorname{an,exp}}}

\newcommand{\diff}{\operatorname{diff}}

\newtheorem{remark}{Remark}

\newtheorem{claim}{Claim}

\newtheorem{assumption}{Assumption}

\newtheorem{definition}{Definition}
\newtheorem{theorem}{Theorem}
\newtheorem{lemma}{Lemma}
\newtheorem{proposition}{Proposition}
\newtheorem{corollary}{Corollary}
\newenvironment{proof}[1][]{\noindent {\bf Proof#1:\;}}{\hfill $\Box$}

\ifpdf
  \DeclareGraphicsExtensions{.eps,.pdf,.png,.jpg}
\else
  \DeclareGraphicsExtensions{.eps}
\fi

\title{Subgradient sampling for nonsmooth nonconvex minimization}

\author{J\'er\^ome Bolte\thanks{Toulouse School of Economics, Universit\'e de Toulouse, France.}
\and Tam Le\footnotemark[2] \and
 Edouard Pauwels \thanks{IRIT, CNRS,  Universit\'e de Toulouse, ANITI, Toulouse France. Institut Universitaire de France.}}

\usepackage{amsopn}

\ifpdf
\hypersetup{
  pdftitle={Subgradient sampling for nonsmooth nonconvex minimization},
  pdfauthor={J. Bolte, T. Le, E. Pauwels}
}
\fi

\begin{document}
\maketitle

\begin{abstract}
 Risk minimization for nonsmooth nonconvex problems naturally leads to first-order sampling or, by an abuse of terminology, to stochastic subgradient descent. We establish the convergence of this method in the path-differentiable case, and describe more precise results under additional geometric assumptions. We recover and improve results from Ermoliev-Norkin \cite{ermol1998stochastic} by using a different approach: conservative calculus and the ODE method.  In the definable case, we show that first-order subgradient sampling avoids artificial critical point with probability one and applies moreover to a large range of risk minimization problems in deep learning, based on the backpropagation oracle.
  As byproducts of our approach, we obtain several results on integration of independent interest, such as an interchange result for conservative derivatives and integrals, or the definability of set-valued parameterized integrals.
\end{abstract}

\section{Introduction}

We consider possibly nonconvex and nonsmooth risk minimization problems of the form 
\begin{equation}
    \label{eq:expectation_minimization}
    \underset{w \in \R^p}{\min} \: \J(w):=\E_{\xi \sim P} \left[f(w, \xi)\right],
\end{equation}
where $P$ is a probability distribution on some measurable space $(S, \mathcal{A})$. We assume throughout the paper that $\mathcal{J}$ is bounded below. This type of problems has many applications, we refer for instance to \cite{Ruszczynski2020,Ruszczynski2021ASS} and references therein, for various examples in several fields. Our specific interest goes in particular to online deep learning \cite{sahoo2017online} and machine learning more broadly \cite{Bottou_2018}.
We consider a minimization approach through first-order sampling: our model is that of the stochastic gradient method\footnote{Also known under the generic acronym {\em SGD}.}, which in its classical form generates iterates $(w_k)_{k \in \N}$ through 
\begin{align}
    \label{eq:algo_intro}
    \left\{
        \begin{array}{ll}
            w_0 \in \R^p  \\
            w_{k+1} = w_k - \alpha_k v(w_k, \xi_k) \hspace{1cm} \text{for } k \in \N,
        \end{array}
    \right.
\end{align}
where $v(w_k, \xi_k)$ is a descent direction at $w_k$ for the function $f(\cdot, \xi_k)$.

When $f$ is smooth, $v$ may be taken to be the gradient of $f$, so that,  in expectation, the search direction boils down to the gradient:
\begin{equation}
    \E_{\xi \sim P} \left[\nabla_w f(\cdot, \xi)\right] =\nabla \E_{\xi \sim P} \left[f(\cdot, \xi) \right]=\nabla \J
    \label{eq:interchangeGradientIntegral}
\end{equation}
where the first equality follows under mild integrability conditions. This is a well-known case for which many convergence results are available, see, e.g., \cite{kushner2003stochastic}. 

In the nonsmooth nonconvex world, it has become classical to consider Clarke subgradient oracles.  In that case the average update direction in \eqref{eq:algo_intro} falls  into $\E_{\xi \sim P} [\partial^c_{w} f(\cdot, \xi)]$; but contrary to \eqref{eq:interchangeGradientIntegral},  
we merely have
 \begin{equation} \label{eq:expectationclarkesubgradient} 
 \E_{\xi \sim P} \left[\partial^c_{w} f(\cdot, \xi) \right] \supset \partial^c \E_{\xi \sim P} \left[f(\cdot, \xi) \right] =\partial^c\J, 
 \end{equation}
see \cite[Theorem 2.7.3]{clarke1990optimization}. Without additional convexity or regularity assumptions, as in \cite{Majewski2018AnalysisON,Davis2020}, the inclusion is strict in general. In plain words, general subgradient sampling is not a stochastic subgradient method: expected increments may not be subgradients of the loss. As a consequence, the very nature of a first-order sampling method may generate undesired directions that could result in sporadically erratic movements or artificial and irrelevant steady states \cite{bolte2020mathematical}. The situation is even worse in deep learning where the oracle is based on backpropagation \cite{Rumelhart1986LearningRB}, which may add more artifacts to the dynamics, see, e.g., \cite{bolte2020mathematical} and references therein. Despite these issues, many real-world algorithms  are designed according to this model, justifying the need for theoretical support and guarantees to the practical success of these methods.

This paper aims at addressing these problems in the vanishing step size regime. We build upon conservative calculus introduced in \cite{bolte2019conservative}. This approach makes rigorous the use of formal subdifferentiation in a wide mathematical framework encompassing most nonconvex nonsmooth problems. One of its main features is that sum or composition of conservative gradients are conservative. Under mild assumptions,  we shall extend this property to integration, parameterized integrals of conservative gradients being conservative. An important advantage of this approach is its full compatibility with one of the core modern algorithms of large-scale optimization: backpropagation \cite{tensorflow2016, pytorch2019_9015, jax2018github}.

An essential series of works on subgradient sampling are those developed in  \cite{ermol1998stochastic}, followed also by \cite{Ruszczynski2020,Ruszczynski2021ASS}, using the generalized derivatives and the calculus introduced in Norkin \cite{norkin1980}, see  \Cref{section:norkin} for more details. 
They address, in particular, the interplay between expectations and subgradients under various assumptions and ensure as well convergence to ``generalized critical points". 
The ideas of \cite{ermol1998stochastic}, and some follow-up research \cite{norkin2021}, have been surprisingly overlooked by the stochastic optimization and machine learning communities, at least until recently\footnote{As for us,  it came to our knowledge in the finalization stage of this paper.}. There is a strong common point between the present article and Norkin's research since, key to our work, is a new first-order calculus. But there are also important differences  and additions that we highlight below:

$-$ We follow the conservative approach of \cite{bolte2019conservative} instead of the generalized derivative approach of \cite{norkin1980}, which are different techniques. Focus on the conservative case is motivated by a growing ``theory" close to machine learning applications: implicit differentiation \cite{bolte2021nonsmooth} with application to implicit neural networks \cite{bai2019}, nondifferentiable programming \cite{jaxopt_implicit_diff}, bi-level programming \cite{pedregosa2016,bertrand2020implicit}), differential equations \cite{marx2022path} which are naturally connected to Neural ODEs \cite{chen2018neural,dupont2019augmented}, or even partial minimization \cite{pauwels2021ridge}. Let us also mention that  conservativity is more general than differentiability in the sense of Norkin, see  \Cref{section:norkin}, even though both notions coincide in the semialgebraic case as recently established in \cite{davis2021conservative}.   

$-$ We provide a general and simple result on the avoidance of artificial critical points. First-order sampling and backpropagation create independently artificial critical points that can swamp the method. Under adequate subanalyticity  assumptions, matching many practical problems, we establish that  this does not occur, see  \Cref{section:geometryofstocopt}.
 Our theory is developed in close connection to semialgebraicity and subanalyticity techniques.  This provides a ``ready-to-use" flavor to our convergence results for online deep learning and also to obtain new results in the definable world. In particular, we provide conditions for set-valued integrands to result in definable parametrized integrals or expectations as well.

$-$ Another specificity of our work is to rely, for its asymptotic analysis, on the ``ODE approach" through the use of Benaim-Hofbauer-Sorin results \cite{benaim} on stochastic approximations.  The   versatility and  simplicity of this method allow for a quite direct extension to other types of first-order methods as, see, e.g., \cite{bianchi2021closed,gadat2018stochastic,castera2021inertial}. Definability  provides simple and explicit sufficient conditions for commonly used abstract hypotheses in this framework, such as path-differentiability of the risk and Sard's condition.

\medskip

The paper is organized as follows: we first present representative samples of our general results in \Cref{section:online_sgd_deeplearning} with an emphasis on applications to online deep learning and the role of semialgebraic, or definable optimization. In \Cref{section:conservativeofintegral}, we present our main theoretical results with, in particular, a theorem of conservative differentiation for integral functions. \Cref{section:geometryofstocopt} gathers results from o-minimal geometry and provides set-valued improvement of Cluckers-Miller's integration result \cite{Cluckers_2011} as well as avoidance result for artificial critical points. \Cref{section:proofresults} contains the proofs of \Cref{section:online_sgd_deeplearning} results.

\paragraph{Notations}  For $q \in \N^*$, $\mathcal{B}(\R^q)$ denotes the Borel sigma algebra on $\R^q$, $\lambda$ is the Lebesgue measure on $\R$. $\|\cdot \|$ denotes the Euclidean norm. For a subset $A$ of a normed vector space, $\conv A$ denotes the convex hull of $A$ and $\bar{A}$ its closure; $\dim A$  denotes its Hausdorff dimension. If in addition, $A$ is bounded then $\|A\| :=\sup\{\|y\| \ |  \ y\in A\}$. For $x \in \R^p$ $r > 0$, $B(x, r)$ is the open ball of center $x$ and radius $r$ with respect to the Euclidean norm.

Given $f : \R^p  \rightarrow \R$ locally Lipschitz continuous, Rademacher's theorem ensures that $f$ is differentiable on a set of full measure denoted $\diff_f$. Its Clarke subgradient is defined for $x \in \R^p$ as
\begin{equation*}
    \partial^c f(x) = \conv \left\{\lim_{k \to + \infty} \nabla f(x_k) \ \middle| \  x_k \in \diff_f, x_k \underset{k \to \infty}{\to} x \right\}.
\end{equation*}
Given  $g : \R^p \times \R^m \rightarrow \R$ such that for almost all $s \in \R^m$ $g(\cdot, s)$ is locally Lipschitz, $\partial^c_w g(\cdot,s)$ denotes the Clarke subgradient of $g(\cdot,s)$ whenever the latter is locally Lipschitz.

For $F: \R^p \rightarrow \R^q$ locally Lipschitz,   the Clarke Jacobian is defined as
\begin{equation*}
    \Jac^c F(x) = \conv \left\{\lim_{k \to + \infty} \Jac F(x_k) \ \middle| \  x_k \in \diff_{F}, x_k \underset{k \to \infty}{\to} x \right\},
\end{equation*}
and for $G : \R^p \times \R^m \rightarrow \R^q$  such that $G(\cdot, s)$ is locally Lipschitz for almost all $s \in \R^m$, $\Jac^c_w G(\cdot,s)$ is  the Clarke Jacobian  of $G(\cdot,s)$ for almost all $s \in \R^m$.

\section{Main results and online deep learning}

This section provides simplified formulations of our results and also gives applications to online deep learning.

\label{section:online_sgd_deeplearning}
\subsection{Subgradient sampling method} \label{subsection:subgradientsamplingmethod}
\paragraph{Framework} Let us consider the stochastic minimization problem (\ref{eq:expectation_minimization}) where $P$ is a fixed probability distribution, $f$ is possibly nonsmooth and nonconvex such that the risk function $\mathcal{J}$ is well defined and bounded below, in particular $f(x, \cdot)$ is $P$-integrable for all $x \in \R^p$. This problem is tackled through the {\em first-order sampling algorithm} (\ref{eq:algo_intro}). $(\xi_k)_{k \in \N}$ is a sequence of i.i.d. random variables with distribution $P$ and the mapping $v : \R^p \times \R^m \rightarrow \R^p$ is a selection of $\partial^c_w f : \R^p \times \R^m \rightrightarrows \R^p$, i.e., for almost all $s \in \R^m$, for all $ w \in \R^p$, $v(w, s) \in \partial^c_w f(w,s).$


\begin{assumption} \label{ass:convergence_essacc}\;
\begin{enumerate}
\item There exist a function $ \kappa : \R^m \rightarrow \R_+$ square integrable with respect to $P$, and $q_0 \in \N$ such that for almost all $s \in \R^m$,
$$ \forall x, y \in \R^p, \: |f(x, s) - f(y, s)| \leq \kappa(s) (1 + \max \left\{\|x\|, \|y\| \right\}^{q_0}) \|x - y \|.$$
    \item For all $k \in \N$, $\alpha_k> 0$, $\sum_{k \in \N} \alpha_k =  \infty$ and $\lim_{k \to \infty} \alpha_k = 0$,
    \item The integrand $f$ and the selection $v$ are semialgebraic,
    \item $\underset{k\in \N}{\sup} \|w_k\| < \infty$ almost surely.
\end{enumerate}
\end{assumption}
Observe that the  first condition implies local Lipschitz continuity of $\mathcal{J}$  hence $\partial^c \mathcal{J}$ is well defined. Furthermore, beyond semialgebraicity, all the results of this work extend to globally subanalytic functions $f$ and  selections $v$, which encompasses the vast majority of machine learning examples. In particular, globally subanalytic functions are polynomially bounded and they satisfy \Cref{ass:convergence_essacc} (1) if $P$ has compact support. 
 
Some solutions have been proposed in the literature to ensure \Cref{ass:convergence_essacc} (4). In \cite{ermol1998stochastic}, the iterates are projected on a compact set. The authors of \cite{borkar, nurminski} use a restart mechanism: the sequence  is reset, either at $w_0$ in \cite{nurminski} or on a sphere in \cite{borkar},  when the objective function is too high. Such mechanisms are however hard to apply in our case since they require computing $\mathcal{J}(w_k)$ at each iteration.

\begin{remark}[Probability spaces] In this work, the term ``almost surely" can refer to different probability spaces.  \Cref{ass:convergence_essacc} (1) refers to a probability space $(S,\mathcal{A},P)$ where $S = \R^m$ and $\mathcal{A} = \mathcal{B}(\R^m)$, whereas in  \Cref{ass:convergence_essacc} (4), it is question of  a coun\-ta\-ble product of probability spaces on which is defined the whole sequence $(w_k)_{k \in \N}$. In the paper, randomness has to be understood regarding these different probability spaces. 
\end{remark}

\paragraph{A glance at the convergence results}
In this general setting, we are not able to guarantee the almost sure convergence of $(w_k)_{k \in \N}$ to the set of critical points of $\mathcal{J}$, but we can describe its limit behavior in a weaker sense introduced in \cite{benaim2000} and extended to set-valued flows in \cite{bianchi2021closed}, see also \cite{faure2011ergodic,Bolte2020LongTD}. This relies on the notion of \textit{essential accumulation points}.

\begin{definition}[Essential accumulation point] {\rm An accumulation point $\bar{w} \in \R^p$ is called {\em essential} if for every neighborhood $U$ of $\bar{w}$ one has
\begin{equation*}
    \limsup_{k \to \infty} \frac{\sum_{i=0}^k \alpha_i \mathbbm{1}_{\left\{w_i \in U\right\}}}{\sum_{i = 0}^k \alpha_i} > 0 \text{ almost surely.}
\end{equation*}}
\end{definition}
Intuitively, non essential accumulation points are hardly never seen. This is made more precise in  \cite[Corollary 4.9]{bianchi2021closed}, through the use of occupation measures.

A first result on the criticality of the essential accumulation points of $(w_k)_{k \in \N}$ is as follows:
\begin{theorem}[Criticality of essential accumulation points]
\label{th:convergence_essacc}
Let $(w_k)_{k \in \N}$ de\-fi\-ned by~(\ref{eq:algo_intro}). Then under \Cref{ass:convergence_essacc}, the following hold:
\begin{enumerate}
    \item \label{th:convergence_essac_1} All essential accumulation points $\bar{w}$ of $(w_k)_{k \in \N}$ satisfy the  weak notion of criticality
\begin{equation}
    \label{eq:clarkeCirticalStochIntSubgrad}
     0 \in \E_{\xi \sim P} \left[\partial^c_w f(\bar{w},\xi) \right] = \int_{\R^m} \partial^c_w f(\bar{w},s) \di P(s)
\end{equation}
where the integral is taken in the sense of Aumann (\Cref{def:aumann}).
    \item \label{th:convergence_essac_2} If $P$ has a density with respect to  Lebesgue measure, there exists a subset $\Gamma \subset \R$ whose complement is finite such that if $\alpha_k \in \Gamma$ for all $k \in \N$, then, for all initialization $w_0$ chosen in a residual  full-measure set, and with probability~$1$, $(w_k)_{k \in \N}$ verifies for all $k \in \N$, 
    \begin{equation*}
						w_{k+1} = w_k - \alpha_k \nabla_w f(w_k,\xi_k) \qquad\mbox{and}\qquad \int_{\R^m} \nabla_w f(w_k, s) dP(s) = \nabla \mathcal{J}(w_k).
    \end{equation*}
		Also, the essential accumulation points $\bar{w}$ of $(w_k)_{k \in \N}$  are Clarke critical, i.e., 
\begin{equation}
\label{eq:clarkeCirticalStochSubgrad}
0 \in \partial^c \mathcal{J}(\bar{w}).
\end{equation}
\end{enumerate}
\end{theorem}

\begin{remark}{\rm
1. {(Generalized criticality)} The criticality notion (\ref{eq:clarkeCirticalStochIntSubgrad}) is inherent to the very nature of subgradient sampling methods and relates to the notion of an artificial critical point as described in \cite{bolte2020mathematical}. Gradient artifacts can be generated by sampling from discrete or continuous distributions as in the two following examples: let $f\colon (w,s) \mapsto  s|w|$ and $P = \frac{1}{2} \delta_{-1} + \frac{1}{2} \delta_1$ or $P$ is the uniform density on $[-1,1]$, then $\partial^c \mathcal{J}(w) = 0$ for all $w \in \R$, but $\E_{\xi \sim P}[\partial^c_w f(0,\xi)] = [-1,1]$. Due to the chain rule characterization of conservative gradients, \Cref{def:conservativeField}. These senseless outputs are however extremely rare, this is why they do not impact first-order optimization much. This is shown in {\em\ref{th:convergence_essac_2}} through \eqref{eq:clarkeCirticalStochSubgrad}
}

2. {(Extension to conservative gradient oracles)} \label{remark:conservativeinsteadofclarke} Under the same assumptions, \Cref{th:convergence_essacc} (\ref{th:convergence_essac_1}) holds, mutatis mutandis, replacing the Clarke subgradient by a conservative gradient. Indeed, $\partial^c_w f$ can be replaced by a semialgebraic set-valued map $D$ where for $P$-almost all $s \in \R^m$, $D(\cdot, s)$ is a conservative gradient for $f(\cdot, s)$. In this case, \Cref{th:convergence_essacc} (\ref{th:convergence_essac_2}) also holds true {\em as is}. 
\end{remark}

With an additional assumption on $(\alpha_k)_{k \in \N}$ and the distribution $P$, all accumulation points are Clarke critical and the risk function converges.
\begin{assumption} \label{ass:convergence_acc}{\rm $P$ has a semialgebraic density with respect to Lebesgue.}
\end{assumption}

\begin{assumption} \label{ass:summabilitystepsizes}
$\sum_{k \in \N} \alpha_k^2 < \infty$.
\end{assumption}

\begin{theorem}[Criticality of accumulation points and convergence] \label{th:algosemialgebraic} Under {\em\Cref{ass:convergence_essacc}-\ref{ass:convergence_acc}-\ref{ass:summabilitystepsizes}},  \Cref{th:convergence_essacc} holds  for all accumulation points. In addition,  $\mathcal{J}(w_k)$ converges almost surely.
\end{theorem}

\begin{remark} 1. (On proofs) In order to prove \cref{th:convergence_essac_2},  we use a property of semialgebraic functions called {\em stratification}, see \Cref{def:stratification}. The intermediary results leading to \cref{th:convergence_essac_2} are gathered in \Cref{section:geometryofstocopt}. The proof of \Cref{th:algosemialgebraic} is based on results from \cite{benaim, bianchi2021closed} and an abstract stochastic approximation version is found in \Cref{section:differentialinclusion}.

2. (On \Cref{ass:convergence_acc})  \label{remark:discretecontinuousdistribution}  The main reason for this assumption is that it provides enough rigidity to ensure a strong form of Sard's theorem for the risk function $\J$. Semialgebraic  densities are extremely flexible: they approximate all continuous densities on compact sets. Although \cref{ass:convergence_acc} relates to the unknown distribution $P$, this is a reasonable proxy for a large class of distributions. Beyond semialgebraicity, $P$ can be assumed to be globally subanalytic, see \Cref{subsection:definablesets} for the definition of such functions, which includes analytic densities with semialgebraic compact support, like  truncated  Gaussian distributions.

Since  definability and consequently Sard's property are preserved by finite sum,  \Cref{th:convergence_essacc} and \cref{th:algosemialgebraic} also hold for $P$ being the joint distribution between a discrete and a continuous random variable. This allows to consider classification tasks in online deep learning (see \cref{corr:convergenceclarke}). We limit ourselves to  an absolutely continuous distribution for the sake of simplicity.

\end{remark}

\paragraph{On calculus: chain rule for expectations}
A pivotal result in our analysis  is the following:
\begin{lemma}[Chain rule for expectations] \label{corr:pathdiffexpectation} Under \Cref{ass:convergence_essacc}, $\mathcal{J}$ admits a chain rule with respect to $\partial^c \mathcal{J}$, i.e., for any absolutely continuous curve $\gamma : \left[0, 1\right] \rightarrow \R^p$,
\begin{equation*}
    \frac{\di }{\di t }(\mathcal{J}\circ \gamma)(t) = \langle a, \dot{\gamma}(t) \rangle \text{ for all } a \in \partial^c \mathcal{J}(\gamma(t)) \text{ for almost all $t \in [0,1]$.}
\end{equation*}
$\mathcal{J}$ also admits a chain rule with respect to $\E_{\xi \sim P}[\partial^c_w f(\cdot, \xi)]$.
\end{lemma}

Let us recall that, functions admitting a chain rule with respect to their Clarke subgradient, or equivalently to a  conservative gradient, are called {\em path-differentiable}. 
 \Cref{corr:pathdiffexpectation} is a consequence of the more general  \Cref{theorem:leibnizconservative} which generalizes the outer sum rule  \cite[Corollary 4]{bolte2019conservative} and extends  \cite[Theorem 1]{norkin1986} to conservative gradients. This result allows to interchange conservative gradient and integral operations: taking the expectation of $v(\cdot,\xi)$ with respect to $\xi \sim P$ gives $\E_{\xi \sim P} \left[v(\cdot, \xi) \right]$, which is a selection in a conservative gradient for $\mathcal{J}$.

\subsection{Comparison with related works}

Recent works \cite{Davis2020,Majewski2018AnalysisON} assume the Clarke subgradient and integral operations to be interchangeable, which in practice requires regularity assumptions. Without such assumptions, as here, first-order sampling  may produce absurd trajectories or lead to spurious critical points. This was first observed  in \cite{ermol1998stochastic} and rediscovered in \cite{bolte2019conservative,bolte2020mathematical}. To avoid converging to these undesirable points \cite[Remark 4.2]{ermol1998stochastic} suggests perturbing  iterates by random noise. 
Instead, our  \Cref{th:algosemialgebraic} ensures that ``almost all" subgradient sampling sequences accumulate to Clarke critical points, justifying the correctness of the algorithm as implemented in practice. Avoidance of spurious critical points is also obtained in \cite{bianchi2020} through probabilistic 
and continuity arguments. The definable framework allows for a sharper description of the set of step sizes and initializations leading to spurious points. For instance, with a finite horizon $K$, the set of bad initializations is a finite union of manifolds of dimension strictly lower than $p$, while the set of bad step sizes is  finite.

Due to the role of the ODE method in the analysis of stochastic optimization methods, considering risk functions having a chain rule is not new in the literature, either as an assumption \cite{bianchi2020} or using restrictive sufficient conditions \cite{Davis2020,Majewski2018AnalysisON}. With  \Cref{corr:pathdiffexpectation}, we provide instead   simple and explicit sufficient conditions (see also \Cref{theorem:leibnizconservative} for a general form). Similarly, our  \Cref{ass:convergence_acc} is sufficient to obtain a strong form of Sard's condition which is usually stated in a weak abstract form as a hypothesis \cite{ermol1998stochastic,benaim,bianchi2020}.

Several works consider a constrained version of  (\ref{eq:expectation_minimization}). For instance, \cite{Schechtman2022} studies a proximal stochastic subgradient method and  \cite{ermol1998stochastic} considers a projection on a compact set. We decide to focus on the unconstrained setting, but a precise study of the constrained setting could be a matter of future work.

\subsection{First-order sampling, backpropagation, and deep learning}
\label{subsection:deeplearning}

\paragraph{Deep learning model} Let $P$ be a probability distribution on  $ \R^d \times \R^I$  called \textit{population} distribution. One of the goals of machine learning is to build a predicting function $h \colon \R^d \mapsto \R^I$ such that $h(X)\simeq Y$ when $(X,Y)$ is distributed according to $P$, written $(X,Y) \sim P$. The input $X$ can be, for instance, an image, and $Y$ can either be discrete for a classification task or continuous for a regression task.

In deep learning, the predictor $h$ is a neural network defined by a compositional structure involving $L$ layers and a parameters vector  $w = (w^{(1)}, \ldots, w^{(L)})$ seen as a vector of $\R^p$. For $l = 0,\ldots, L$, the $l$-th layer is represented by a real vector $x^{(l)} \in \R^{p_l}$. We consider that layer $0$ has the same dimension as $\R^d$, $p_0 = d$. Given an input $x \in \R^d$ the predicting function encoded by the neural network with parameter $w \in \R^p$ is denoted by $h(w, \cdot)$ and is defined by the relations:
\begin{align}
    \left\{
        \begin{array}{ll}
            x^{(0)} &= x,  \\
            x^{(l)} &= g_l(w^{(l)}, x^{(l-1)}), \hspace{1cm} \text{for } l = 1, \ldots, L,\\
            h(w, x) &= x^{(L)}.
        \end{array}
    \right.
    \label{eq:neuralNetwork}
\end{align}

For $l = 1, \ldots, L$, the function $g_l$ can take several forms, see e.g. \cite{lecun2015deeplearning}.
With a slight abuse of notation, we will consider that $w$ is a vector in $\R^p$, consisting of the concatenation of all vectors $w_1, \ldots, w_L$.

\paragraph{Online deep learning}  Given a loss function $\ell \colon \R^I \times \R^I \to \R $, training is formulated as a risk minimization problem:
\begin{align}
     \min_{w \in \R^p} \ \mathcal{J}(w) := \E_{(X, Y) \sim P} \left[ \ell(h(w,X), Y) \right]
    \label{eq:neuralNetTraining}
\end{align}

The expectation is generally unknown and approximated through statistical sampling. Consider  a sequence  $\left( x_k, y_k\right)_{k \in \N}$ of i.i.d. samples generated from the distribution $P$ and, let us tackle the problem \eqref{eq:neuralNetTraining} with the first-order sampling algorithm (\ref{eq:algo_intro}). In this setting, $(\xi_k)_{k \in \N}$ is $(x_k, y_k)_{k \in \N}$ while $f$ is the function $(w, x,y) \mapsto \ell(h(w,x), y)$. 

\paragraph{Backpropagation and conservative gradients} In deep learning first-order information is accessed using backpropagation \cite{Rumelhart1986LearningRB}. It is an efficient application of the chain rule of differentiable calculus which provides a numerical evaluation of the derivative of $f$. We will consider differentiation with respect to the decision variable $w$ in \eqref{eq:neuralNetTraining} and denote the output of backpropagation by $\backprop_w$. The function $f$ writes as a composition $f = f_r \circ \ldots \circ f_1$ where the functions $f_1, \ldots, f_r$ involve the functions $\ell$ and $g_1, \ldots, g_L$ from (\ref{eq:neuralNetwork}) and (\ref{eq:neuralNetTraining}). When applied to nondifferentiable functions $f_1, \ldots, f_r$, $\backprop_w$ could be considered as an oracle evaluating an element in the product of their Clarke Jacobians:
\begin{align}
    \backprop_w f(w, x,y) \nonumber
    & \in \partial^c f_r(f_{r-1} \circ \ldots \circ f_1(w, x,y) )^T \\ &\times \Jac^c f_{r-1}(f_{r-2} \circ \ldots \circ f_1(w, x,y)) \times \ldots \times \Jac^c_w f_1(w, x, y).
    \label{eq:defBackprop}
\end{align}

With this definition in mind, we may define the backpropagation variant of (\ref{eq:algo_intro}).
\begin{algorithm}[ht]
	\caption{First-order sampling with backpropagation} 
	\label{alg:onlineSGDdeepLearning}
	\begin{algorithmic}[1]
	\STATE \textbf{Inputs}: \\
	$w_0 \in \R^p$, $\left( x_k, y_k\right)_{k \in \N}$ i.i.d. with distribution $P$, $(\alpha_k)_{k\in \N}$ positive step sizes.
		\FOR {$k=0,1,\ldots$}
            \STATE $w_{k+1} = w_k - \alpha_k \backprop_{w_k}\ell(h(w_k,x_k), y_k)$ 
		\ENDFOR
	\end{algorithmic} 
\end{algorithm}

Backpropagation does not necessarily compute an element of the Clarke subgradient. When $f_1, \ldots, f_r$ are path-differentiable, backpropagation  computes a selection of a conservative gradient of $f$. It is satisfied for instance if  $\ell, g_1, \ldots ,g_L$ are locally Lipschitz and semialgebraic, or more generally, globally subanalytic. We hence assume the following which is a  condition satisfied by the vast majority of applications:
\begin{assumption}[Locally Lipschitz continuity and semialgebraicity] \label{ass:lipschitzandsemialgebraic}{\rm
	The neural network training problem in \eqref{eq:neuralNetTraining} satisfies for $l = 1 ,\ldots, L$, the functions $g_l \colon \R^{p_{l-1}} \to \R^{p_l}$ and $\ell \colon \R^I \times \R^I \to \R$ are locally Lipschitz and semialgebraic functions.
	\label{ass:neuralNetStructure}}
\end{assumption}

In order to formulate our results, we further require an assumption on the  distribution $P$. This assumption is quite mild as it encompasses a large class of probability distributions in classification or regression.
\begin{assumption}[Semialgebraic distribution with compact support]{\rm
	The joint distribution $P$ satisfies one of the following:}

	        $-$ {\rm Regression: $P$ has a semialgebraic density $\phi$ with compact support with respect to Lebesgue on $\R^d \times \R^I$.}

					$-$ {\rm Classification: $P$ is supported on $\R^d \times (e_i)_{i=1}^I$ where for $i = 1 \ldots I$, $e_i$ is the $i$-th element of the canonical basis in $\R^I$. The distribution $P$ factorizes  as $P(X,Y) = P(Y) P(X|Y)$ and we assume $P_Y$ is discrete over $(e_i)_{i=1}^I$ 
									and $P(X|Y = e_i)$ has a density  $\phi_i$ with respect to Lebesgue, that is semialgebraic and compactly supported.}
	\label{ass:datadistribution_dl}
\end{assumption}

In this setting, a consequence of \Cref{th:convergence_essacc} and \Cref{theorem:subsequentialconvergence} is the following:
\begin{theorem}[First-order sampling and training for online deep learning]
\label{corr:convergenceclarke} Under Assumptions \ref{ass:neuralNetStructure} and \ref{ass:datadistribution_dl}, let $(w_k)_{k \in \N}$ be generated by \Cref{alg:onlineSGDdeepLearning}. We suppose that the sequence $(\alpha_k)_{k \in \N }$ is strictly positive and verifies $\sum_{k \in \N} \alpha_k =  \infty$, $ \alpha_k = o(1/\log(k))$. Assume furthermore that  $\sup \{ \|w_k\|:  k \in \N\}<\infty$ almost surely. 

Then there exist a set $\Gamma \subset \R$ whose complement is finite, and $W \subset \R^p$ of full measure and residual such that if for all $k \in \N, \alpha_k \in \Gamma$ and $w_0 \in W$, $\mathcal{J}(w_k)$ converges almost surely as $k \to \infty$ and all accumulation point $\bar{w}$ of $(w_k)_{k \in \N}$ is Clarke critical, i.e., verifies $0 \in \partial^c \mathcal{J}(\bar{w})$.
\end{theorem}

This theorem parallels \cite[Theorem 9]{bolte2019conservative} for a general population distribution: assumptions are simple and widespread, it is fully compatible with backpropagation, and shows that spurious outputs hardly matter for the training phase. 
Note however that the validity of absolutely continuous randomness is questionable in practice. In particular, the density of $W$, and the finiteness of $\Gamma^c$ can be affected due to quantization errors. For instance, \cite{NEURIPS2021_043ab21f}  empirically shows that gradient artifacts created by conservative calculus can impact the method.

Similar models have been considered in the literature, but with somehow stronger assumptions, e.g., \cite{Ruszczynski2021ASS} (which rules out nonsmooth activation functions). Another important contribution is given in \cite{norkin2021}, as discussed in the introduction.

\section{Nonsmooth analysis for stochastic approximation algorithms}
\label{section:conservativeofintegral}

\subsection{Set-valued analysis and conservative gradients}
\label{subsection:conservativefields_setvalued}

Let $(X, \mathcal{F})$ be a measurable space. Let us recall some results from set-valued analysis.

\begin{definition}[Measurable set-valued maps]{\rm Denote $\mathcal{K}_{p}$ the set of nonempty compact subsets of~$\R^p$. It is a measurable space considering the Borel $\sigma$-algebra $\mathcal{B}_H(\mathcal{K}_p)$ induced by the topology of the Hausdorff  distance. A nonempty compact valued map $F : X \rightrightarrows \R^p$  is called {\em measurable}, if it is measurable from $(X, \mathcal{F})$ to $(\mathcal{K}_{p}, \mathcal{B}_H(\mathcal{K}_p) )$. In this case, for all closed subsets $A \subset \R^p$, the upper inverse $F^u(A) := \left\{x \in X \ | \ F(x) \subset A \right\}$ is measurable in $(X, \mathcal{F})$.}
\end{definition}

\begin{proposition}[Measurable selections, Theorem 18.13 \cite{infinite}] \label{prop:measurableselection} Let $F : X \rightrightarrows \R^p$ be a  measurable  nonempty and compact valued map. Then there exists a measurable selection of $F$, that is a measurable function $v : X \rightarrow \R^p$ satisfying for all $x \in X$, $v(x) \in F(x)$.
\end{proposition}

\begin{corollary}[Castaing's Theorem] \label{corollary:castaing} Let $F : X \rightrightarrows \R^p$ nonempty compact valued. Then $F$ is measurable if and only if there exists a sequence of measurable selections $\left(F_n\right)_{n \in \N}$ such that $\forall x \in X, F(x) = \overline{\left\{F_1(x), F_2(x),... \right\}}$.
\end{corollary}

\begin{remark}[Measurability of set-valued composition] \label{remark:castaing}{\rm \Cref{corollary:castaing} can be used to justify measurability of composed set-valued functions. For instance, given $g : \R^p \rightarrow \R$ continuous and $F : X \rightrightarrows \R^p$ compact valued and measurable, then $g \circ F$ is measurable. Indeed, let $\left(F_k\right)_{k \in \N}$ be a sequence of measurable selections given by Castaing's Theorem, such that $\forall x \in X, F(s) = \overline{\left\{F_1(x), F_2(x),... \right\}}$. Then by continuity of $g$, we have for all $x \in X,$ $g(F(x)) = g(\overline{\left\{F_1(x), F_2(x),... \right\}}) = \overline{\left\{g(F_1(x)), g(F_2(x)),... \right\}}$. The functions $g \circ F_i $ are all measurable and $g \circ F$ is compact valued by continuity of $g$, whence by Castaing's Theorem, $g \circ F$ is measurable.}
\end{remark}

\begin{definition}[Aumann integral] \label{def:aumann}{\rm Let $(X, \mathcal{F}, \mu)$ be a measure space and $F : X \rightrightarrows \R^p$ a set-valued map. Then the {\em integral} of $F$ with respect to the measure $\mu$ is
\begin{equation*}
    \int_X F(x) \di \mu(x) = \left\{\int_X v(x) \di \mu(x) \ \middle|\  v \text{ is integrable and for all } x \in X, v(x) \in F(x)\right\}.
\end{equation*}
}
\end{definition}

\smallskip

We  also use the expectation notation $\E_{\xi \sim P} \left[F(\xi)\right]=\int_X F(x) \di P(x)$ whenever $(X, \mathcal{F}, P)$ is a probability space and $\xi$ is a random variable with distribution $P$.

\begin{definition}{\rm Let $F : \R^m \rightrightarrows \R^p$ be a set-valued map.}
\begin{enumerate}   
\item (Graph closedness)  $F$  is {\em graph closed } if its graph defined by $\Graph F := \left\{(x,y) \in \R^m \times \R^p \ | \ y \in F(x) \right\}$,
is a closed subset of $\R^m \times \R^p$.
\item (Local boundedness) $F$ is  {\em locally bounded} if for all $x \in \R^m$, there exist a neighborhood $\mathcal{U}$ of $x$ and $M > 0$, such that for all $z \in U$ and $y \in F(z)$, $\|y\| < M$.  
\item (Upper semicontinuity) $F$ is  {\em upper semicontinuous} at $x \in \R^m$, if for each open subset $\mathcal{V}$ containing $F(x)$, there exists a neighborhood $\mathcal{U}$ of $x$ such that for all $z \in \mathcal{U}$, $F(z) \subset \mathcal{V}$.
\end{enumerate}
\end{definition}

\begin{definition}[Conservative gradient]{\rm Let $f : \R^p \rightarrow \R$ be locally Lipschitz continuous. A locally bounded and graph closed set-valued map $D : \R^p \rightrightarrows \R^p$ that is nonempty valued is called {\em a conservative  gradient} for $f$, if for all absolutely continuous curve $\gamma : [0, 1]  \rightarrow \R^p$, $f$ admits a chain rule with respect to $D$ along $\gamma$, i.e.,
\begin{equation*}
        \frac{\di}{\di t}(f \circ \gamma)(t) = \langle v,\dot{\gamma}(t)\rangle, \text{ for all }  v \in D(\gamma(t)) \text{  and almost all $t \in \left[0, 1\right]$.}
\end{equation*}}
\label{def:conservativeField}
\end{definition}
Lipschitz continuous functions $f:\R^p\to\R$ admitting a conservative gradient are called  \textit{path-differentiable}. They are central to our analysis.

\subsection{Conservative gradient of integral functions}

In this part $(S, \mathcal{A}, \mu)$ is a complete measure space. We consider a function $f : \R^p \times S \xrightarrow[]{} \R$ such that for almost all $s \in S$, $f(\cdot, s)$ is path-differentiable with conservative gradient $D(\cdot, s)$. Our goal is to show a result of ``conservative differentiation" under the integral sign  in order to get a conservative calculus for the parametrized integral  $\int_{S} f(\cdot, s) \di \mu(s)$. 

First, we provide a result of derivation under the integral sign when the integrand is absolutely continuous in its first variable. We shall use the following lemma.
\begin{lemma}[Measurability of partial derivatives] \label{derivabilitydomain} Let $U \subset \R$ open and $f : U \times S \xrightarrow[]{} \R$ a $(\mathcal{B}(\R) \times \mathcal{A})$-measurable function. We suppose that there exists  $M \subset S$ of full measure such that for all $s \in M$, $f(\cdot, s)$ is absolutely continuous. Then $\frac{\partial f}{\partial x}$ is jointly measurable and is defined almost everywhere in $U \times S$. Also, for almost all $x \in U$, $\frac{\partial f}{\partial x}(x, s)$ is defined for almost all $s \in S$.
\end{lemma}
\begin{proof}
Define the following quantities for all $x \in U$ and $s \in M$:
\begin{equation*}
    f'_u(x, s) = \limsup_{h\rightarrow 0} \frac{f(x + h, s) - f(x, s)}{h} \text{ and }    f'_l(x, s) = \liminf_{h\rightarrow 0} \frac{f(x + h, s) - f(x, s)}{h}.
\end{equation*}
By continuity of $f$ both limit operators may operate only  in $\Q$ without changing the value of $f'_u$ and $f'_l$. Whence $f'_l$ and $f'_u$ are measurable and so is $\frac{\partial f}{\partial x}$. Furthermore, the  domain $E$ of $\frac{\partial f}{\partial x}$  is  $\left\{(x, s) \in U \times S \ \middle| \ f'_l(x, s) = f'_u(x, s), \: - \infty < f_u(x, s) < +\infty \right\}$ which is measurable. Applying Fubini's Theorem yields
\begin{equation*}
    \int_{U \times S} \mathbbm{1}_{E^c}(x, s) \di (\lambda \times \mu) \:(x, s) = \int_{S} \int_U \mathbbm{1}_{E^c}(x, s) \di x \di \mu( s) = \int_{U} \int_{S} \mathbbm{1}_{E^c}(x, s) \di \mu(s) \di x.
\end{equation*}
Since $f(\cdot, s)$ is absolutely continuous for $s \in M$, it is differentiable almost everywhere, hence  $\forall s \in M, \int_U \mathbbm{1}_{E^c}(x, s) \di x = 0$ and the second integral is zero. The third integral vanishes, so  for almost all $x \in U$, $\int_{S} \mathbbm{1}_{E^c}(x, s) \di \mu( s )= 0$, i.e., $\frac{\partial f}{\partial x}(x, s)$ is defined for almost all $s$, which concludes the proof.
\end{proof}

\begin{proposition}[Differentiation of absolutely continuous integrals] \label{leibniz} Let $U \subset \R$ open and $f : U \times S \xrightarrow[]{} \R$ such that:

\begin{enumerate}
    \item For all $x \in U$, $f(x, \cdot)$ is integrable.
    \item For almost all $s \in S$, $f(\cdot, s)$ is absolutely continuous.
    \item $\frac{\partial f}{\partial x}$ is locally integrable, jointly in $x$ and $s$:  for any compact interval $\left[a, b\right] \subset U$,
    \begin{equation*}
        \int_{S} \int_a^b  \left|\frac{\partial f}{\partial x}(x, s)  \right| \di x \di \mu(s)  < \infty.
    \end{equation*}
\end{enumerate}
Then, the function $g : x \mapsto \int_{S} f(x, s) \di \mu(s)$, is absolutely continuous, differentiable at almost all $x \in U$ with $g'(x) = \int_{S} \frac{\partial f}{\partial x}(x, s) \di \mu(s)$.
\end{proposition}
\begin{proof} Let $f : U \times S \xrightarrow[]{} \R$ satisfying all the assumptions. We consider the function $g : x \in U \mapsto \int_{S} f(x, s) \di \mu(s)$ and $a < b$ in $U$. From \Cref{derivabilitydomain}, $\frac{\partial f}{\partial x}(x, s)$ and exists a.e. in $(x, s) \in U \times S$ and admits a measurable extension. The a.e. defined function $\frac{\partial f}{\partial x}$ is identified with some measurable extension. Since for almost all $s \in S$ $f(\cdot, s)$ is absolutely continuous, the fundamental theorem of calculus for Lebesgue integration (see Theorem 14 in Section 4, Chapter 5 of \cite{royden1968real}) implies that 
\begin{equation*}
    g(b) - g(a) = \int_{S} \left[ f(b, s) - f(a, s) \right] \di \mu(s)
                = \int_{S} \int_{a}^b \frac{\partial f}{\partial t}(t, s) \di t \di \mu(s).
\end{equation*}
Under Assumption 3, by measure completeness,  Fubini-Lebesgue's Theorem applies: $ g(b) - g(a) = \int_{a}^b \int_{S} \frac{\partial f}{\partial t}(t, s) \di \mu(s) \di t$.
The function $x \mapsto \int_{S} \frac{\partial f}{\partial x}(x, s) \di \mu(s)$ is integrable on $\left[a, b\right]$ so $g$ is absolutely continuous. By the fundamental theorem of calculus, $g'(x)$ is defined for almost all $x \in U$ with $g'(x) = \int_{S} \frac{\partial f}{\partial x}(x, s) \di \mu(s)$.
\end{proof}

\bigskip

The following result is one of the cornerstones of this paper. It can be seen as interchanging conservative gradient and integral operations.

\begin{theorem}[Path-differentiability of parametrized integrals] \label{theorem:leibnizconservative} Let $D :\R^p \times S \rightrightarrows \R^p$ and $f : \R^p \times S \xrightarrow[]{} \R$ such that:  
\begin{enumerate}
    \item For all $x \in \R^p$, $f(x, \cdot)$ is integrable.
    \item For almost all $s \in S$, $f( \cdot, s)$ is locally Lipschitz continuous and $D(\cdot, s)$ is conservative for $f( \cdot, s)$.
    \item \label{item:jointmeasurable} $D : \R^p \times S \rightrightarrows \R^p$ is jointly measurable in $\mathcal{B}(\R^p) \times \mathcal{A}$.
    \item \label{item:boundedconservative} For all compact subset $C \subset \R^p$, there exists an integrable function $\kappa : S \xrightarrow[]{} \R_{+}$ such that for all $ (x, s) \in C \times S,\ \|D(x, s)\| \leq \kappa(s)$, where for $(x, s) \in \R^p \times S$, $\|D(x, s)\|:= \underset{y \in D(x, s)}{\sup} \: \|y\|$.
\end{enumerate}
Then $\int_{S} f(\cdot, s) \di \mu(s)$ is path-differentiable and $\int_{S} D(\cdot, s) \di \mu( s)$ is a conservative gradient for $\int_{S} f(\cdot, s) \di \mu(s)$.
\end{theorem}

\begin{proof}
With the chain rule definition of conservativity, \Cref{def:conservativeField}, we will show that the problem reduces to the differentiation of an absolutely continuous integral and then, we shall use \Cref{leibniz}  to conclude. 
Let $f : \R^p \times S \rightarrow \R$ and $D : \R^p \rightrightarrows \R^p$ verifying the assumptions 1 to 4 displayed above.

Following the conservative gradient definition, \Cref{def:conservativeField}, in particular we verify $\int_{S} D(\cdot, s) \di \mu(s)$ is graph closed, nonempty valued and locally bounded. By the measurable selection theorem, see \Cref{prop:measurableselection}, $\int_S D(\cdot, s) \di \mu(s)$ is nonempty valued. It is locally bounded by item 4. For almost all $s \in S$, $D(\cdot, s)$ is graph closed and locally bounded, hence it is upper semicontinuous by \cite[Corollary 1 in Chapter 1, Section 1]{aubin}.
  By Aumann's integral properties, see \cite[Theorem 2]{shapiro}, and since $D(\cdot, s)$ is upper semicontinuous, compact valued for all $s$, then $\int_{S} D(\cdot, s) \di\mu( s)$ is graph closed. 

Now, we have to verify the chain rule property. Let $\gamma : \left[0, 1\right] \xrightarrow[]{} \R^p$ be any absolutely continuous curve. By hypothesis, there exists a set of full measure $M \subset S$ such that for all $s \in M$, $f(\cdot, s)$ has conservative gradient $D(\cdot, s)$. We have $\forall s \in M$, $f(\gamma(\cdot), s)$ is absolutely continuous because $f$ is locally Lipschitz in $x \in C$ and $\gamma$ is absolutely continuous. Thus, $\forall s \in M$ $f(\gamma(\cdot), s)$ is differentiable a.e. and the chain rule property (\ref{eq:chainrulealmosteverywhere}) holds for almost all $t \in \left[0, 1\right]$, i.e.,
\begin{equation}
    \label{eq:chainrulealmosteverywhere}
    \forall v \in D(\gamma(t), s), \: \frac{\di}{\di t} f(\gamma(t), s) = \langle v, \dot{\gamma}(t) \rangle.
\end{equation}
Let $E\subset [0,1]\times S$ be the domain of existence of $\frac{\di}{\di t} f(\gamma(t), s)$. $E$ is measurable and of full measure according to \Cref{derivabilitydomain}. We want to verify the measurability of the domain of validity of \cref{eq:chainrulealmosteverywhere}, which is
$E \: \cap \: \left\{(t, s) \in [0, 1] \times S \ | \ \varphi(t, s) = 0 \right\}$, where  $\varphi(t, s) =  \frac{\di}{\di t} f(\gamma(t), s) - \langle D(\gamma(t), s), \dot{\gamma}(t) \rangle$ for all $ (t, s) \in E$ and $\varphi(t,s) = 1$ elsewhere. By Castaing's Theorem (see \Cref{remark:castaing}) $\varphi$ is measurable. The set $\left\{(t, s) \in [0, 1] \times S \ | \ \varphi(t, s) = 0 \right\}$ is exactly the upper inverse of $\left\{0\right\}$ by $\varphi$, $\varphi^u(\left\{ 0 \right\})$ hence it is jointly measurable in $(\R \times S, \mathcal{B}(\R) \otimes \mathcal{A})$. Similarly as in the proof of \Cref{derivabilitydomain}, by Fubini's Theorem,  $\varphi^u(\left\{ 0 \right\})$ is of full measure and there exists $I_1 \subset \left[0, 1\right]$ of full measure such that for all $ t \in I$ \cref{eq:chainrulealmosteverywhere} holds for almost all $s \in S$.

Let $t \in I_1$. From \cref{eq:chainrulealmosteverywhere}, we can say that, for any measurable selection $v : s \xrightarrow[]{} \R^p$ of $D(\gamma(t), \cdot)$, we have for almost all $s \in S$
\begin{equation}
    \label{eq:integratingselection}
    \frac{\di}{\di t} f(\gamma(t), s) = \langle v(s), \dot{\gamma}(t) \rangle.
\end{equation}
Integrating (\ref{eq:integratingselection}) over $s \in S$ we have for any $a$ in the Aumann integral $\int_{S} D(\gamma(t), s) \di s$ and measurable selection $v$ such that $a = \int_{S} v(s) \di s$,
\begin{equation}
    \label{eq:integratingselection2}
    \int_{S} \frac{\di}{\di t} f(\gamma(t), s) \di\mu(s) = \int_{S}\langle v(s), \dot{\gamma}(t)  \rangle \di \mu(s) = \langle a, \dot{\gamma}(t) \rangle.
\end{equation}
On the other hand, $\gamma(\left[ 0, 1\right])$ is compact  by continuity of $\gamma$. Let $\kappa$ given by assumption 4 for the compact set $C = \gamma([0,1])$. The Cauchy-Schwarz inequality gives for all $(t, s) \in \left[0, 1\right] \times S$, 
$|\langle v(s), \dot{\gamma}(t) \rangle| \leq \|D(\gamma(t), s)\| \|\dot{\gamma}(t) \|\leq \kappa(s) \|\dot{\gamma}(t)\|$. 
Since $\gamma$ is absolutely continuous, $\dot{\gamma}$ is integrable on $\left[0, 1\right]$ hence the function $(t, s) \mapsto \kappa(s)\| \dot{\gamma}(t) \|$ is locally integrable jointly in $(t,s)$, and so is $(t, s) \mapsto \frac{\di}{\di t}f(\gamma(t), s) = \langle v(s), \dot{\gamma}(t)\rangle$.
\Cref{leibniz} now applies to $(t,s) \mapsto f(\gamma(t), s)$, hence there exists $I_2$ of full measure such that
\begin{equation}
    \label{leibnizstep}
    \forall t \in I_2,\: \int_{S} \frac{\di}{\di t} f(\gamma(t), s) \di \mu(s) =  \frac{\di}{\di t} \int_{S}  f(\gamma(t), s) \di\mu(s).
\end{equation}
Combining \cref{eq:integratingselection2} which holds on $I_1$ and \cref{leibnizstep} which holds on $I_2$ we have
\begin{equation*}
    \forall t \in I_1 \cap I_2, \:\forall  a \in \int_{S} D(\gamma(t), s) \di \mu(s),\: \frac{\di}{\di t} \int_{S}  f(\gamma(t), s) \di \mu(s) = \langle a, \dot{\gamma}(t) \rangle \hspace{1cm}
\end{equation*}
and $ I_1 \cap I_2$ is of full measure. Finally we have shown that $\int_{S} D(\cdot, s) \di \mu(s)$ is nonempty, compact, valued graph closed, and verifies the chain rule property, hence it is a conservative gradient for $\int_{S} f(\cdot, s) \di \mu(s)$.
\end{proof}

\subsection{Application to stochastic approximation}
\label{section:differentialinclusion}

Let $P$ be a probability measure on $(S, \mathcal{A})$,  denote $\operatorname{supp} P$ its support, and consider a jointly measurable set-valued map $D : \R^p \times S \rightrightarrows \R^p$ such that for almost all $s \in S$, $f( \cdot, s)$ is locally Lipschitz and $D(\cdot, s)$ is a convex-valued conservative gradient for $f( \cdot, s)$. We consider the sequence $(w_k)_{k \in \N}$ defined by (\ref{eq:algo_intro}) where $v$ is now a measurable selection of $D$, i.e., for all $(w,s) \in \R^p \times S$, $v(w, s) \in D(w,s)$. For all $k \in \N^*$, the truncated random sequence  $(w_0, \ldots, w_k)$ is defined on the product probability space $(S^k, \mathcal{A}^{\otimes \N}, P^{\otimes k})$, and the whole trajectory $(w_k)_{k \in \N}$ is defined on the countable product $(S^{\N}, \mathcal{A}^{\otimes k }, P^{\otimes \N})$. When there is no ambiguity, the term ``almost sure" implicitly refers  to these spaces.

In order to study the sequence $(w_k)_{k \in \N}$, we  use the nonsmooth ODE methods developed in \cite{benaim, bianchi2021closed}. To this end, we consider the set-valued map
\begin{equation*}
    D_{\mathcal{J}} : w \mapsto \E_{\xi \sim P} [D(w,\xi)]
\end{equation*}
and the differential inclusion
\begin{equation}
    \label{eq:differentialinclusion}
    \dot{w} \in -D_{\mathcal{J}}(w).
\end{equation}
A \textit{solution} to the differential inclusion (\ref{eq:differentialinclusion}) with initial point $w_0 \in \R^p$ is an absolutely continuous curve $w : \R_+ \xrightarrow[]{} \R^p$ such that $w(0) = w_0$ and for almost all $t \in \R_+$, $\dot{w}(t) \in -D_{\mathcal{J}}(w(t))$. 

\begin{remark}[Solution to the differential inclusion on $\R_+$]
	\label{rem:solutionFlow} 	Since our analysis is based on a nonsmooth ODE method, the differential inclusion (\ref{eq:differentialinclusion}) must admit solutions. Available theory \cite{Filippov1988DifferentialEW,aubin} requires a convex valued conservative gradient $D$, which is satisfied for $D = \partial^c_w f$ or replacing $D$ by its convex hull.
	
	The flow of the differential inclusion (\ref{eq:differentialinclusion}) is then nonempty and well defined on $\R_+$ as $\mathcal{J}$ in \eqref{eq:expectation_minimization} is assumed to be bounded below. 
	Indeed, consider a strict lower bound $\mathcal{J}^* \in \R$. 
	Fix $T > 0$ and $w_0 \in \R^p$,  let $\mathcal{D} = [-T,T] \times \overline{B(w_0, r)}$ with $r := 2T\sqrt{(\mathcal{J}(w_0) - \mathcal{J}^*)}>0$, which is closed, bounded domain and contains $w_0$ and $t = 0$. 
	Since $D_\mathcal{J}$ is a conservative gradient, it is nonempty valued, and graph closed, locally bounded hence compact valued and upper semicontinuous. 
	It is furthermore convex valued by convexity of $D$ since set-valued integration preserves convexity. 

	We are in the setting of \cite[Chapter 2, Section 7]{Filippov1988DifferentialEW}, and by Theorem 1 in \cite[Chapter 2, Section 7]{Filippov1988DifferentialEW}, there exists $d > 0$ with $d < T$ such that (\ref{eq:differentialinclusion}) admits at least a solution on $[0,d]$. Let $w$ with $w(0)= w_0$ be an arbitrary such solution. By Theorem 2 \cite[Chapter 2, Section 7]{Filippov1988DifferentialEW},  $w$ can be continued to the boundary of $\mathcal{D}$, i.e., to $t = T$ or $\| w_0 - w(t) \| = r$. Since $D_\mathcal{J}$ is a conservative gradient for $\mathcal{J}$, one has for all $t \in [0, d]$,
    \begin{align*}
        \left(\frac{1}{t} \int_0^t \|\dot{w}(s)\|  \di s \right)^2  \leq \frac{1}{t}  \int_0^t \|\dot{w}(s)\|^2 \di s &=- \int_0^t \langle D_{\mathcal{J}}(w(s) ), \dot{w}(s) \rangle \di s   \\ &= \mathcal{J}(w_0) - \mathcal{J}(w(t))
    \end{align*}
	so that $\|w_0 - w(t)\| \leq  \int_0^t \|\dot{w}(s)\|  \di s  \leq t\sqrt{(\mathcal{J}(w_0) - \mathcal{J}^*)}$. However $t\sqrt{(\mathcal{J}(w_0) - \mathcal{J}^*) }< r$, hence the solution $w$ can be continued to $t = T$. 
	Since $w_0$, $T$ and $w$ were arbitrary, the flow of (\ref{eq:differentialinclusion}) is nonempty and well defined on $\R_+$.

\end{remark}

We define the {\em set-valued flow} $\Phi_t$, given at all $w_0 \in \R^p$ and $t \in \R_+$ as
\begin{equation*}
\Phi_t(w_0) := \left\{w(t) \ | \ w : \R_+ \rightarrow \R^p \text{ is a solution of (\ref{eq:differentialinclusion}) with } w(0) = w_0 \right\}.
\end{equation*}
We also recall the definition of a Lyapunov function for a set-valued flow from \cite{benaim}:
\begin{definition}[Lyapunov function for a set]{\rm A continuous function  $\mathcal{F}$ is a {\em Lyapunov function for a set $S\subset \R^p$} and for the dynamical system (\ref{eq:differentialinclusion}) if
\begin{align*}
    &\forall x \in \R^p \setminus S\;, \forall t > 0, \forall y \in \Phi_t(x), \:\mathcal{F}(y) < \mathcal{F}(x), \\ &\forall x \in S,\; \forall t \geq 0, \forall y \in \Phi_t(x), \:\mathcal{F}(y) \leq \mathcal{F}(x).
\end{align*}
}
\end{definition}

We define furthermore the critical set associated to $D_{\mathcal{J}}$, that is, $\crit D_\mathcal{J} := \left\{x \in \R^p \ | \ 0 \in D_{\mathcal{J}}(x) \right\}$, and show the following property:

\begin{lemma}[Conservative gradient and Lyapunov function] \label{lemma:lyapunovcondition} Let $D_\mathcal{J}$ a conservative gradient for $\mathcal{J}$. Then $\mathcal{J}$ is a Lyapunov function for $\crit D_{\mathcal{J}}$ and the differential inclusion $\dot{w} \in - D_{\mathcal{J}}(w)$.
\end{lemma}

\begin{proof}
Let $x\in \R^p$, $t \geq 0$ and $y \in \Phi_t(x)$. By definition of $\Phi_t$, there exists $w : \R_+ \rightarrow \R^p$ a solution to the differential inclusion $\dot{w} \in - D_{\mathcal{J}}(w)$ with initial value $w(0) = x \in \R^p$ such that $y = w(t)$. By definition of a conservative gradient and since $w$ is absolutely continuous, we have:
\begin{equation}
    \label{eq:lyapunovproof}
    \mathcal{J}(w(t)) - \mathcal{J}(w(0)) = \int_{0}^t \langle D_{\mathcal{J}}(w(u)), \dot{w}(u) \rangle \di u.
\end{equation}
Since $w$ is a solution of (1), $\dot{w}(u) \in - D_{\mathcal{J}}(w(u))$ for almost all $u \in \left[ 0, t\right]$ and  we have $\mathcal{J}(w(t)) - \mathcal{J}(w(0)) = - \int_{0}^t \| \dot{w}(u) \|^2 \di u$, hence $\mathcal{J}(w(t)) = \mathcal{J}(y) \leq \mathcal{J}(x)$. 

Now we suppose that $x \in \R^p \setminus \operatorname{crit} D_{\mathcal{J}}$ and $t > 0$. By upper semicontinuity of $D_{\mathcal{J}}$, $\exists \epsilon > 0, \exists \delta > 0, \forall y \in \R^p \text{ such that } \| y - x\| \leq \delta,$ we have $ \forall v \in D_{\mathcal{J}}(y), \| v \| \geq \epsilon$. By continuity of $w$, $\exists t_0 > 0, \forall u \in \left[0, t_0\right], \|w(u) - x\| \leq \delta $ hence $\|\dot{w}(u) \| \geq \epsilon$ for almost all $u \in \left[0, t_0\right]$. Thus, by integration, $\int_{0}^t \| \dot{w}(u) \|^2 \di u$ is strictly positive and $\mathcal{J}(y) < \mathcal{J}(x)$.
\end{proof}

\begin{assumption} Let $P$ be a probability measure on $(S, \mathcal{A})$, consider $D : \R^p \times S \rightrightarrows \R^p$ such that for almost all $s \in S$, $f( \cdot, s)$ is locally Lipschitz continuous and $D(\cdot, s)$ is a convex-valued conservative gradient for $f( \cdot, s)$ and:
\begin{enumerate}
    \label{ass:integralconservative}
    \item For all $w \in \R^p$, $f(w, \cdot)$ is integrable with respect to $P$. 
    \item $D : \R^p \times S \rightrightarrows \R^p$ is jointly measurable in $\R^p \times S$.
    \item There exist a  function $\kappa : S \xrightarrow[]{} \R_{+}$ square integrable with respect to $P$, and $q_0 \in \N$ such that for $P$-almost all $s \in S$, for all $w  \in \R^p$,  $\|D(w, s)\| \leq \kappa(s) (1 + \|w\|^{q_0})$, where $\|D(w, s)\|:= \underset{y \in D(w, s)}{\sup} \: \|y\|$.
\end{enumerate}
\end{assumption}
\Cref{ass:integralconservative} parallels assumptions of \Cref{theorem:leibnizconservative}. In particular, \Cref{ass:integralconservative} (3) implies  (4) in \Cref{theorem:leibnizconservative}. In \Cref{ass:integralconservative} (3), the term $(1 + \|w\|^{q_0})$ can be replaced by a locally bounded function $\psi(w)$. For simplicity we assume a polynomial boundedness which also coincides with semialgebraicity of \Cref{ass:convergence_essacc}.

Consider furthermore the following assumption for the algorithmic recursion \eqref{eq:algo_intro}:

\begin{assumption}
\label{ass:stepsizeandsequenceboundedness}
				In addition to \Cref{ass:integralconservative}, the sequence $(\alpha_k)_{k \in \N}$ is strictly positive,  $\sum_{k \in \N} \alpha_k = \infty$ and $\alpha_k \to 0$ as $k \to \infty$. $v$ jointly measurable in $\mathcal{B}(\R^p) \times \mathcal{A}$ is such that for all $w \in \R^p$, for almost all $s \in S$, $v(w,s) \in D(w,s)$. The event $\underset{k \in \N}{\sup} \: \|w_k\|  < \infty$ occurs almost surely.
				\label{ass:algorithmicAssumptionStochApproximation}
\end{assumption}

\begin{lemma}[Noise extinction]
				For the recursion \eqref{eq:algo_intro} under \Cref{ass:algorithmicAssumptionStochApproximation}, set for all $k \in \N$, $a_k = \E [ v(w_k, \xi_k) | w_k]$  and $u_k = v(w_k, \xi_k) - a_k$, then
				\begin{enumerate}
								\item $\sup_{k \in \N} \E[\|u_k\|^2|w_k]$ is finite almost surely.
								\item If $\sum_{k \in \N} \alpha_k^2 < \infty$, then $\sum_{i=0}^k \alpha_i u_i$ converges almost surely as $k \to \infty$.
								\item If $ \phi \colon w \mapsto  \sup_{s \in \operatorname{supp} P} \|v(w,s)\|$ is locally bounded and $\alpha_k = o(1/\log(k))$, then
								\begin{align}
								    \forall T> 0, \qquad \lim_{k \to \infty} \sup_{m\geq k} \left\{ \left\|\sum_{i=k}^m \alpha_{i} u_{i}\right\|,\quad \mathrm{s.t. }\sum_{i=k}^m \alpha_{i} \leq T\right\} = 0\qquad a.s.
								    \label{eq:noiseSummability}
								\end{align}
				\end{enumerate}
				\label{lem:summabilityNoise}
\end{lemma}
\begin{proof}
	With these notations, (\ref{eq:algo_intro}) writes $w_{k+1} = w_k - \alpha_k(a_k + u_k )$  for all $k \in \N$,
	where we have by joint measurability of $v$, $a_k \in \E_{\xi \sim P}[ D(w_k, \xi)] = D_{\mathcal{J}}(w_k)$.

	As for the first statement, we have for all $k \in \N$,
	\begin{align}
	    \E\left[\| u_k \|^2 \ | \ w_k \right] &= \E_{\xi \sim P} \left[\| v(w_k,\xi) - a_k \|^2 \right] \leq \E_{\xi \sim P} \left[(\|v(w_k, \xi)\|+ \|a_k\|)^2 \right] \nonumber\\ 
	    & \leq 4 \E_{\xi \sim P}\left[\|D(w_k, \xi)\|^2 \right]  \leq 4\E_{\xi \sim P} \left[\kappa(\xi)^2 \right] (1 + \|w_k\|^{q_0})^2,
			\label{eq:supNoiseBounded}
	\end{align}	
	so that $\sup_{k \in \N} \E[\|u_k\|^2|w_k] \leq 4\E_{\xi \sim P} \left[\kappa(\xi)^2 \right] (1 + R^{q_0})^2$ where $\kappa$ given by \Cref{ass:integralconservative} is square integrable and $R := \sup_{k \in \N} \|w_k\|$ is almost surely finite by \Cref{ass:stepsizeandsequenceboundedness}. This proves the first statement.

    Toward a proof of the second statement, let $\epsilon_k := \alpha_k u_k$ and $Z_k := \sum_{i=0}^k \epsilon_i$ for $k \in \N$. We want to prove $Z_k$ converges almost surely as $k \to \infty$. Fix an arbitrary $M > 0$, and define $\epsilon_k^M := \alpha_k u_k \mathbbm{1}_{\left\{ \|w_k\| \leq M\right\}}$ and $Z^M_k := \sum_{i=0}^k \epsilon_i^M$. $(\epsilon_k^M)_{k \in \N}$ is a martingale difference sequence. Indeed, for $k \in \N$, by independence of $\xi_k$ we have $\E \left[u_k \ | \ w_k \right] = \E_{\xi \sim P}[ v(w_k, \xi)] - a_k = 0$, hence $\E \left[\epsilon_k^M |  w_k \right] = 0$ and $(Z_k^M)_{k \in \N}$ is a martingale relatively to the filtration generated by $(\xi_k)_{k \in \N}$. We will apply a martingale convergence theorem \cite[Theorem 4.5.2]{durrett2010probability} on $(Z_k^M)_{k \in \N}$. We have to verify $\E[\|Z_k^M\|^2] < \infty$ and $\sum_{k = 0}^{\infty} \E\left[\| \epsilon_k^M \|^2 \ \middle| \ w_k \right] < \infty$.
    By the inequality (\ref{eq:supNoiseBounded}), one has 
    \begin{align*}
        \E[\|\epsilon_k^M \|^2 | w_k]  & = \alpha^2_k \mathbbm{1}_{\left\{\|w_k\| \leq M \right\}}  \E[\|u_k \|^2 | w_k] \\
        & \leq 4 \alpha^2_k \mathbbm{1}_{\left\{\|w_k\| \leq M \right\}} \E_{\xi \sim P} \left[\kappa(\xi)^2 \right] (1 + \|w_k\|^{q_0})^2 \\
        & \leq 4\alpha^2_k \E_{\xi \sim P} \left[\kappa(\xi)^2 \right]  (1 + M^{q_0})^2,
    \end{align*}
    and taking the expectation gives $\E[\|\epsilon_k^M\|^2] < \infty$, hence $\E [\| Z_k^M \|^2] < \infty$ for all $k \in \N$. We have furthermore 
    $$\sum_{k = 0}^{\infty} \E\left[\| \epsilon_k^M \|^2 \ \middle| \ w_k \right] \leq 4 \E_{\xi \sim P} \left[\kappa(\xi)^2 \right]  (1 + M^{q_0})^2 \sum_{k=0}^\infty \alpha_k^2   < \infty$$ since we assumed $\sum^{\infty}_{k=0} \alpha_k^2 < \infty$ for the second statement. Finally, \cite[Theorem 4.5.2]{durrett2010probability} applies and $\sum_{i=0}^k \epsilon^M_i$  converges almost surely as $k \to \infty$. In particular, 
    \begin{equation}
        \label{eq:truncatedsequence}
        P^{\otimes \N}(\{\sum_{i=0}^\infty \epsilon^M_i \text{converges} \} \cap \{ \sup_{k \in \N} \| w_k\| \leq M \} ) =  P^{\otimes \N}(\sup_{k \in \N} \| w_k\| \leq M).
    \end{equation}
    We can finally prove the almost sure convergence of $Z_k$.
	\begin{align*}
	    P^{\otimes \N}( \{\sum_{i=0}^\infty \epsilon_i \text{ converges} \}) & = P^{\otimes \N}(\{\sum_{i=0}^\infty \epsilon_i \text{ converges} \}\cap\{ \sup_{k \in \N} \| w_k\| < \infty\}) \\
	    & = \lim_{M \to \infty} P^{\otimes \N}(\{\sum_{i=0}^\infty \epsilon_i \text{ converges} \} \cap  \{ \sup_{k \in \N} \| w_k\| \leq M \}) \\
	    & = \lim_{M \to \infty} P^{\otimes \N}(\{\sum_{i=0}^\infty \epsilon_i^M \text{converges} \} \cap \{ \sup_{k \in \N} \|w_k\| \leq M \}) \\
	    & = \lim_{M \to \infty} P^{\otimes \N}(\sup_{k \in \N} \| w_k\| \leq M) \\
	    & = P^{\otimes \N}(\sup_{k \in \N} \|w_k\| < \infty) = 1.
	\end{align*}
The first and the last equalities follow from $\sup_{k \in \N} \|w_k\| < \infty$ in  \Cref{ass:stepsizeandsequenceboundedness}. The fourth equality follows from (\ref{eq:truncatedsequence}). The third equality follows from the fact that for $i \in \N$, $\epsilon^M_i = \epsilon_i$ whenever $\sup_{k \in \N} \|w_k\| \leq M$. The second and fifth equalities are obtained by monotone convergence since the event $\sup_{k \in \N} \|w_k \| \leq M$ is increasing with respect to $M$, and $\bigcup_{M = 0}^\infty \left\{ \sup_{k \in \N} \|w_k \| \leq M \right\} = \left\{\sup_{k \in \N} \|w_k \| < \infty\right\}$.

	For the third statement, let for all $k \in \N$, $c_k = 2\max\left\{1, \max_{i=0,\ldots,k} \phi(w_k)\right\}$, almost surely increasing and convergent, such that $\| u_k / c_k \| \leq 1$ almost surely. Note that $u_k/c_k$ are martingale increments: for $k \in \N$, $u_k/c_k$ is integrable and  $1/c_k$ is measurable with respect to $w_0, \ldots, w_{k}$ hence $\mathbb{E}[u_k / c_k| w_0,\ldots, w_k] = \mathbb{E}[u_k | w_0,\ldots, w_k] / c_k = 0$, see \cite[Theorem 4.1.14]{durrett2010probability}. Fix $c > 0$, we have $2\log(k) = o( c/ \alpha_k)$ so that as $k \to \infty$, $\exp\left(-c/\alpha_k\right)  k^2 = \exp( -c/\alpha_k + 2 \log(k)) \to 0$
    and $\exp(-c/\alpha_k) = o(1/k^2)$ is summable. We invoke \cite[Proposition 4.4]{benaim1999dynamics}:
    \begin{align*}
	    \forall T> 0 \qquad \lim_{k \to \infty} \sup_m \left\{ \left\|\sum_{i=k}^m \alpha_{i} u_{i} / c_i\right\|,\quad \mathrm{s.t. }\sum_{i=k}^m \alpha_{i} \leq T\right\} = 0\qquad a.s.
	\end{align*}
	Note that \cite[Proposition 4.4]{benaim1999dynamics} can be applied to any subgaussian martingale difference sequence, here we apply it to $(u_k/c_k)_{k \in \N}$  uniformly bounded by $1$, hence subgaussian. Fix $T> 0$ and set for all $k$,  $m_k$ the largest integer $m \geq k$ such that $\sum_{i=k}^m \alpha_i \leq T$. 
    We now have for all $k \leq m \leq m_k$ 
    \begin{align*}
	    &\left\|\sum_{i=k}^m \frac{\alpha_{i} u_{i}}{c_i} - \frac{1}{c_k} \sum_{i=k}^m \alpha_{i} u_{i} \right\| = \left\|\sum_{i=k}^m \left(\frac{1}{c_i} - \frac{1}{c_k} \right)\alpha_{i} u_{i} \right\| \\
	    \leq \quad& \left(\frac{1}{c_k} - \frac{1}{c_{m_k}} \right)  \sum_{i=k}^{m} \alpha_i c_{i} \leq \left(\frac{1}{c_k} - \frac{1}{c_{m_k}} \right) c_{m_k} \sum_{i=k}^{m_k} \alpha_i \leq  \left(\frac{1}{c_k} - \frac{1}{c_{m_k}} \right)c_{m_k} T,
	\end{align*}
	and the result follows because $c_k$ converges almost surely.
\end{proof}

As in \cite{benaim}, see also \cite[Section 6]{shikhman}, let us consider a Morse-Sard assumption in order to state the central contribution of this work in terms of stochastic approximation. 
\begin{assumption}[Morse-Sard]{\rm
    \label{ass:morsesard}
    $\mathcal{J}(\operatorname{crit} D_{\mathcal{J}})$ has empty interior.}
\end{assumption}
\begin{theorem}[Convergence of the subgradient sampling method] \label{theorem:subsequentialconvergence} 
				Let \Cref{ass:algorithmicAssumptionStochApproximation} holds, then almost surely all essential accumulation points $\bar{w}$ of $(w_k)_{k \in \N}$ satisfy $0 \in D_{\mathcal{J}}(\bar{w})$.
				If in addition \Cref{ass:morsesard} and \Cref{lem:summabilityNoise} (2) or (3) hold, then, almost surely, $\mathcal{J}(w_k)$ converges as $k \to \infty$ and all accumulation points $\bar{w}$ of $(w_k)_{k \in \N}$ satisfy $0 \in D_{\mathcal{J}}(\bar{w})$. Furthermore, the set of accumulation points is connected.
\end{theorem}
\begin{proof}
				The first statement is a consequence of \cite[Corollary 4.9]{bianchi2021closed} with $D_\mathcal{J}$. Indeed, Assumptions 4.1 and 4.3 of \cite[Corollary 4.9]{bianchi2021closed} obviously hold, and Assumption 4.2 is the first point of Lemma \ref{lem:summabilityNoise}. Under \Cref{ass:integralconservative}, \Cref{theorem:leibnizconservative} applies and $D_{\mathcal{J}}$ is a conservative gradient for $\mathcal{J}$. \Cref{lemma:lyapunovcondition} applies and $\mathcal{J}$ is Lyapunov for $\operatorname{crit} D_\mathcal{J}$ and the result follows. Note that it would suffice to merely ensure that $\alpha_k \to 0$ to obtain this result.

				Let us now prove the second statement. In case 2 or 3 
				of \Cref{lem:summabilityNoise}, identity \eqref{eq:noiseSummability} holds which is (i) in \cite[Proposition 1.3]{benaim}, and (ii) also holds by \Cref{ass:stepsizeandsequenceboundedness}. Let $w: \R_+ \rightarrow \R^p$ be the affine interpolation of $(w_k)_{k \in \N}$ (see \cite[Definition IV]{benaim} for its formal construction, this is a central object, also used for example in \cite{benaim2000,faure2011ergodic,Bolte2020LongTD,bianchi2021closed}). 
				\cite[Proposition 1.3]{benaim} ensures that $w$ is a perturbed solution to (\ref{eq:differentialinclusion}) almost surely. By \cite[Theorem 4.2]{benaim} $w$ satisfies \cite[Theorem 4.1 (ii)]{benaim}  for (\ref{eq:differentialinclusion}) which implies by \cite[Theorem 4.3]{benaim} that the limit set of $w$ is internally chain transitive with probability $1$.
Furthermore, $\mathcal{J}$ is a Lyapunov function for $\operatorname{crit} D_{\mathcal{J}}$ and the differential inclusion (\ref{eq:differentialinclusion}). By \Cref{ass:morsesard}, $\mathcal{J}(\operatorname{crit} {D_\mathcal{J}})$ has empty interior.  All the conditions are then satisfied to apply  \cite[Proposition 3.27]{benaim}, hence almost surely the limit set of $w$ is  contained in $\operatorname{crit} {D_\mathcal{J}}$ and $\mathcal{J}$ is constant on the limit set of $w$.  We  remark that the limit set of $w$ is equal to the set of accumulation points of $(w_k)_{k \in \N}$, which gives the desired result.

If \Cref{lem:summabilityNoise} (2) or (3) holds, then $\|w_{k+1} - w_k \| \to 0$ almost surely as $k \to \infty$ hence, in this case, accumulation points of $(w_k)_{k \in \N}$ is a connected subset of $\crit D_{\mathcal{J}}$.
\end{proof} 

\section{On the geometry of stochastic optimization pro\-blems}
\label{section:geometryofstocopt}

\subsection{Definable sets and functions}

\label{subsection:definablesets}

\begin{definition}[o-minimal structure] {\rm 
\label{def:ominimal}
Let $\mathcal{O} = (\mathcal{O}_p)_{p \in \mathbb{N}}$ be a collection of sets such that, for all $p \in \N$, $\mathcal{O}_p$ is a set of subsets of $\R^p$. $\mathcal{O}$ is an o-minimal structure on $(\mathbb{R}, +, \cdot )$ if it satisfies the following axioms, for all $p \in \N$:
    
1.  $\mathcal{O}_p$ is stable by finite intersection, union, complementation, and contains $\R^p$. 

2. If $A \in \mathcal{O}_p$ then $A \times \mathbb{R}$ and $\mathbb{R} \times A$ belong to $\mathcal{O}_{p+1}$.

3. If $A \in \mathcal{O}_{p+1}$ then $\pi(A) \in \mathcal{O}_p$, where $\pi$ projects on the $p$ first coordinates,.

4. $\mathcal{O}_p$ contains 
all sets of the form $\left\{x \in \mathbb{R}^p : P(x) = 0\right\}$, where $P$ is a polynomial. 

5. The elements of $\mathcal{O}_1$ are exactly the finite unions of intervals.
}
\end{definition}

\begin{definition}[Definable set, function and set-valued map]{\rm If $\mathcal{O} = (\mathcal{O}_p)_{p \in \N}$ is an o-minimal structure, a subset $A$ of $\R^n$ with $n \geq 1$ is said to be {\em definable in $\mathcal{O}$} if $A \in \mathcal{O}_n$. A function $f : \R^p \rightarrow \R^q$ and a set-valued map $F : \R^p \rightrightarrows \R^q$  are called definable (in an o-minimal structure) if their graphs are definable, as sets of $\R^{p + q}$.}
\end{definition}

The definition allows in particular to prove definability through the use of  first-order formula (see \cite{coste} for more details).
We now present the main examples of o-minimal structures we use in this paper. The following definitions apply by extension to functions through their graph as above.

The class of semialgebraic sets is the smallest o-minimal structure.
\begin{definition}[Semialgebraic sets] {\rm 
	A subset $A \subset \R^n$ is {\em semialgebraic} if there exist polynomial functions $P_{ij}$ and $Q_{ij}$ with $i = 1, \ldots, l$ and $j =1, \ldots, k$ such that $A =  \bigcup_{i = 1}^l \bigcap_{j=1}^k\: \{x \in \mathbb{R}^n \ | \ P_{ij}(x) < 0, \ Q_{ij}(x) = 0 \}$.}
\end{definition}

\paragraph{Globally subanalytic sets, $\Ran$} An important o-minimal structure containing analytic functions is that of {\em globally subanalytic sets}, denoted $\Ran$. In order to define it, we first recall the definitions of semianalytic and subanalytic sets

\begin{definition}[Semianalytic sets \cite{vandendries1996}]{\rm (i) (Semianalyticity) A subset $A$ of $\R^n$ is semianalytic if for any point $x \in \R^d$, there exists a neighborhood $U$ of $x$ such that $U \cap A$ has the form $\bigcup_{i = 1}^l \bigcap_{j=1}^k\: \{x \in \mathbb{R}^n \ | \ g_{ij}(x) < 0, \ h_{ij}(x) = 0$, where the $g_{ij}$ and $h_{ij}$ are real analytic.
(ii) (Subanalyticity) A subset $A$ of $\R^n$ is {\em subanalytic} if there exists $m \in \N$ such that $A$ is  the projection of a semianalytic set $M \subset \R^{n + m} $ on $\R^n$.}
\end{definition}

\begin{definition}[Globally subanalytic sets]
    $A \subset \R^n$ is globally subanalytic if $\tau_n(A)$ is subanalytic, where $\tau_n : x \mapsto \left( \frac{x_1}{\sqrt{1 + x_1^2}}, \ldots, \frac{x_n}{\sqrt{1 + x_n^2}} \right)$.
\end{definition}
\paragraph{O-minimal structure $\Ranexp$} The exponential function, yet widely used in machine learning, is not globally subanalytic. It is however definable in an o-minimal structure called $\Ranexp$ which also contains $\Ran$, see \cite{Dries1995OnTR}.

We now recall important properties of definable sets and functions.
\begin{proposition}[Definable choice \cite{coste}] \label{prop:definablechoice} Let $A \subset \R^p \times \R^q$ a definable set. Denote $\mathcal{P}_p$ the projection on the $p$ first coordinates. Then there exists a definable function $h : \mathcal{P}_pA \rightarrow \R^q$ such that for all $x \in \mathcal{P}_pA$, $(x, h(x)) \in A$.
\end{proposition}


\begin{definition}[$C^r$-stratification \cite{vandendries1996}] \label{def:stratification}{\rm 
Given $A \subset \R^p$ definable set,  $(M_i)_{i = 1, \ldots, n}$ is a {\em $C^r$-stratification} if $(M_i)_{i = 1, \ldots, n}$ is a partition of $A$, and for all $i,j$ in $\{1, \ldots, r\}$, $\overline{M_i} \cap M_j \neq \emptyset$ implies that $M_j$ is included in the boundary of $M_i$.}
\end{definition}

Definable sets admits a $C^r$-stratification for all $r \in \N$. Take for instance the graph of the absolute value  $| \cdot |$, which is a semialgebraic set. It admits as a stratification $\{M_1, M_2, M_3\}$ where $M_1 = \{(x, -x) \ | \ x \in ]- \infty, 0[\}$, $M_2 = \{(x, x) \ | \ x \in ]0, +\infty[\}$, and $M_3 = \{(0,0)\}$. This stratification property has several useful consequences for our work. For a definable function $F : \R^p \rightarrow \R$, one  may consider a stratification $(M'_i)_{i = 1, \ldots n'}$ of $\R^p$ in which for all $i = 1, \ldots, n'$, the restriction of $F$ to $M'_i$ is $C^r$. Taking the  union of the manifolds $M_i'$ of dimension $p$ gives a  dense open subset in $\R^p$ on which $F$ is $C^r$. We use this property with $r  = 2$, for instance in the proof of \Cref{claim:preimageImproved}. The stratification property also implies that definable dense sets have full measure. This is a property we repeatedly use in  \cref{subsection:artifactsavoidance}.

\subsection{Definability and set-valued integration}
\label{subsection:definableintegration}
The following result is a set-valued version of  \cite[Theorem 1.3]{Cluckers_2011}.

\begin{theorem}[Definable set-valued integrals] \label{proposition:definabilityofsetvaluedintegral} Let $\phi$ a globally subanalytic density function on $\R^m$, $D : \R^p \times \R^m \rightrightarrows \R^p$ globally subanalytic, graph closed, locally bounded and convex valued. Assume $D_\mathcal{J} : w \mapsto \int_{\R^m} D(w, s) \phi(s) \di s $ is well defined, then $D_{\mathcal{J}}$ is definable in $\Ranexp$.
\end{theorem}

\begin{proof}
Let $D$ be as in the theorem. We will apply \cite[Theorem 1.3]{Cluckers_2011}  to prove the definability of $D_\mathcal{J}$ in $\Ranexp$. $D$ is definable in $\Ran$ hence the set-valued map $(w, q, s) \mapsto \langle D(w,s), q \rangle$ is definable in $\Ran$ as well. Let $G$ be its graph. The function $H : (w, q, s)
    \mapsto  \max_{v \in D(w, s)} \langle v, q\rangle$
 is clearly definable in $\Ran$ as its graph writes 
$$\operatorname{Graph} H = \left\{(w,q,s,y) \in G \ | \ \forall (w',q', s', y') \in G, (w',q', s') = (w, q , s) \implies y \geq y' \right\}.$$

By definition of the Aumann integral, we have  for all $w \in \R^p$
$$ D_\mathcal{J}(w)  = \left\{\int_{\R^m} g(s) \phi(s) \di s \ \middle| \ g \text{ is a measurable selection of } D(w, \cdot) \right\},$$

For $C \subset \R^p$ compact convex, define the support function $h_C : q \mapsto \max_{v \in C} \langle v, q \rangle$. Then by linearity of the integral, for $(w, q) \in \R^p \times \R^p$, it holds that
\begin{align*}
h_{D_\mathcal{J}(w)}(q) &= \underset{v \in D_\mathcal{J}(w)}{\max}  \langle v, q \rangle \\ &= \max \left\{\int_{\R^m} \langle g(s), q \rangle  \phi(s) \di s \ \middle | \ g \text{ is a measurable selection of } D(w, \cdot) \right\}.
\end{align*}

By \cite[Theorem 18.19]{infinite}, there exists a measurable selection $\Tilde{g} : \R^m \rightarrow \R^p$ of $D(w, \cdot)$ such that $\forall s \in \R^m, \langle \Tilde{g}(s), q \rangle = \max_{v \in D(w, s)} \langle v, q\rangle = H(w,q,s)$, and thus, $\Tilde{g}$ achieves the maximum $h_{D_\mathcal{J}(w)}(q)$, i.e., $h_{D_\mathcal{J}(w)}(q) = \int_{\R^m} H(w,q,s) \phi(s) \di s$. Furthermore, by duality, for all convex compact set $C \subset \R^p$ it holds that $C = \{z \in \R^p \ | \ \sup_{v \in \R^p} \langle v, z \rangle - h_C(v) = 0\}$. Applying this property with $C = D_{\mathcal{J}}(w)$ for each $w \in \R^p$ gives
\begin{align}
    \operatorname{Graph} D_\mathcal{J} &= \left\{(w, z) \in \R^p \times \R^p \ \middle| \ z \in D_\mathcal{J}(w) \right\} \nonumber \\
    &= \left\{(w, z) \in \R^p \times \R^p \ \middle| \ \underset{q \in \R^p}{\sup} \langle q, z \rangle - h_{D_\mathcal{J}(w)}(q) = 0 \right\} \nonumber \\
    &= \left\{(w, z) \in \R^p \times \R^p \ \middle| \ \underset{q \in \R^p}{\sup} \langle q, z \rangle -  \int_{\R^m} H(w,q,s) \phi(s) \di s = 0 \right\}.
    \label{eq:graphDJ}
\end{align}

By \cite[Theorem 1.3]{Cluckers_2011}, $(w, q) \in \R^p \times \R^p \mapsto \int_{\R^m} H(w,q,s) \phi(s) \di s$ is definable in $\Ranexp$,  hence by equality (\ref{eq:graphDJ}) $D_{\mathcal{J}}$ is also definable in $\Ranexp$. 
\end{proof}

\subsection{Consequences in stochastic optimization}
\label{subsection:artifactsavoidance}

\paragraph{Preliminary results}
Before providing our stochastic results, let us establish some technical lemmas. For a definable set $L \subset \R^p \times \R^m$, we define for all $(w,s) \in \R^p \times \R^m$, $L_w := \left\{s \in \R^m \ | \ (w,s) \in L \right\}$ and $L_s := \left\{w \in \R^p \ | \ (w,s) \in L \right\}$.
\begin{lemma}
    \label{lem:fubiniDefinable}
    Let $(r, q) \in \N^* \times \N^*$, and $L \subset \R^r \times \R^q$ be a definable set. Then $L$ is dense if and only if there is a dense definable set $Z \subset \R^r$ such that for all $w \in Z$, $L_w$ is dense in $\R^q$.
\end{lemma}
\begin{proof}
		Let us start with the direct implication.
    Set 
    $Z = \{w \in \R^r \ | \ \forall z \in \R^q, \forall \epsilon > 0, \exists s \in \R^q, (w,s) \in L, \|s - z\| < \epsilon\}.$
    This set is definable and is precisely the set of $w$ such that $L_w$ is dense in $\R^q$. Assume that $Z^c$ has a nonempty interior. This means that there is a nonempty open set $U \subset \R^r$, such that for all $w \in U$
     $$\exists z \in \R^q, \exists \epsilon > 0, \forall s \in \R^q, (w,s) \in L \Rightarrow \|s - z\| \geq \epsilon.$$
     By definable choice, see \Cref{prop:definablechoice}, there are definable functions $z \colon U \to \R^q$ and $\epsilon \colon U \to \R^*_+$, such that for all $w \in U$, we have $ \{(w,v)\in U \times \R^q \ | \ \|v - z(w)\| < \epsilon(w) \} \subset L^c$.
     By stratification, see \Cref{def:stratification}, reducing $U$ if needed, $z$ and $\epsilon$ can be chosen continuous hence $L^c$ has nonempty interior which contradicts the density of $L$. 

		 As for the reverse implication, fix any $(\bar{w},\bar{s}) \in  \R^r \times \R^q$ and $\epsilon >0$. Since $Z$ is dense there is $w \in Z$ such that $\|w - \bar{w}\| < \epsilon/\sqrt{2}$. Since $L_w$ is dense, there is $s \in L_w$ such that $\|s - \bar{s}\| < \epsilon/\sqrt{2}$. Overall, we have $(w,s) \in L$ such that $\|(w,s) - (\bar{w},\bar{s})\| < \epsilon$ which shows that $L$ is dense as $(\bar{w},\bar{s}) \in  \R^r \times \R^q$ and $\epsilon >0$ were arbitrary. 
\end{proof}

\begin{claim}
\label{claim:preimageImproved}
Let $g\colon  \R^p \times \R^m \rightarrow \R^p$ be a definable function. Then there exists a definable dense open set $L \subset \R^p \times \R^m$, a subset $\Gamma \subset \R$ which complement is finite as well as a definable dense set $\Delta \subset \Gamma \times \R^m$, such that $g$ is $C^2$ on $L$, for every $\alpha \in \Gamma$, the definable set $\{s \in \R^m \ | \ (\alpha,s) \in \Delta\}$ is dense open in $\R^m$ and for all $(\alpha,s) \in \Delta$, $L_s$ is dense and open. Furthermore, denoting  $\Phi_{\alpha, s} =  \operatorname{Id} - \alpha \nabla_w g( \cdot, s)$ from  $L_s$ dense open to $\R^p$, we have
\begin{equation*}
     \forall Z \subset \R^p \text{ definable},\ \dim Z \leq p-1 \Rightarrow \dim \Phi_{\alpha, s}^{-1}(Z) \leq p-1.
\end{equation*}
\end{claim}
\begin{proof}
Denote by $L$ a definable dense open set such that $g$ is $C^2$ on $L$ (such sets exist by stratification, see \Cref{def:stratification}). Let $\lambda : L \subset \R^p \times \R^m \rightarrow \R^p$ be a definable representation of the eigenvalues of $\nabla_w^2 g$, where $\nabla^2_w g$ denotes the partial Hessian of $g$ with respect to the variable $w$. Refine $L$ so that $\lambda$ is jointly differentiable in $(w,s)$.   $L$ is open and dense by definability of $g$.  We further set 
$S_0 \subset \R^m$
the definable dense set obtained from \Cref{lem:fubiniDefinable} such that for all $s \in S_0$ the set $L_s$ is open dense in $\R^p$.

Let $F$ be the complement of the critical values of the function $\lambda_i$, for $i = 1,\ldots, p$ on $L$. The set of critical values $F^c$ is finite by the definable Sard's theorem \cite{bolte2007clarke}. Set $\Gamma :=\left\{\alpha \in \R \ | \ \alpha \neq 0, \alpha^{-1} \in F \right\}$. For $i = 1,\ldots,p$, set $$E_{i} := \left\{(\alpha,s) \in \Gamma \times S_0 \ | \ \exists w \in L, \alpha \lambda_i(w,s) = 1, \nabla_w\lambda_i(w,s) = 0\right\}.$$ This set is definable because it is defined by a first-order formula involving definable functions and $L$, $F$, $S_0$ which are definable sets. Let us fix an arbitrary $\alpha \in \Gamma$, and show that the set $E_{\alpha,i} := \{s \in \R^m \ | \  (\alpha,s) \in E_i\}$ has empty interior. By definable choice, \Cref{prop:definablechoice}, there exists $\Tilde{w} : E_{\alpha,i}\rightarrow \R^p$ definable, such that $\forall s \in E_{\alpha,i} $, $\alpha \lambda_i(\Tilde{w}(s), s) = 1$ and $\nabla_w \lambda_i(\Tilde{w}(s), s) = 0$. Assume for the sake of contradiction that there exists a nonempty open subset $U \subset E_{\alpha,i}$. By definability of $\Tilde{w}$ and stratification, $U$ can be chosen so that $\Tilde{w}$ is continuously differentiable on $U$. Then denoting $\Tilde{\lambda}_i : s \mapsto \lambda_i(\Tilde{w}(s), s)$ we have for all $s \in U$, $\nabla \Tilde{\lambda}_i(s) = 0$. The chain rule applied on $\Tilde{\lambda}_i$ yields 
\begin{equation*}
    \forall s \in U, \ \nabla \Tilde{\lambda}_i(s) =  \Jac {\Tilde{w}}(s)^\mathsf{T} \nabla_w \lambda_i (\Tilde{w}(s),s) + \nabla_s \lambda_i(\Tilde{w}(s),s)  = \nabla_s \lambda_i(\Tilde{w}(s),s) = 0,  
\end{equation*}
hence we have for all $s \in U$, $\nabla \lambda_i(\Tilde{w}(s), s) = 0$. In other words, since $\lambda_i(\Tilde{w}(s), s ) = \alpha^{-1}$ for all $s \in U$, then $\alpha^{-1}$ is a critical value of $\lambda_i$ which contradicts $\alpha \in \Gamma$. This shows that $E_{\alpha,i}$ has empty interior for all $\alpha$ in $\Gamma$, therefore $E_i$ also has empty interior. 

Set $\Delta = \left(\bigcup_{i=1}^p E_{i} \right)^c$, $\Delta$ is the complement of a finite union of definable sets with empty interiors so it is definable and dense. \Cref{lem:fubiniDefinable} implies that there are only  finitely many values $\alpha$ such that $\{s \in \R^m \ | \  (\alpha,s) \in \Delta\}$ is not dense in $\R^m$. Therefore, we may refine further $\Gamma$ by removing finitely many points, and refine $\Delta$ accordingly such that it satisfies the desired projection property: for every $\alpha \in \Gamma$, the set $\{s \in \R^m \ | \  (\alpha,s) \in \Delta\}$ is dense in $\R^m$.

Now, fix $\alpha \in \Gamma$ and $s$ such that $(\alpha,s) \in \Delta$. Consider the set 
\begin{equation*}
    K_{\alpha, s} = \left\{ w \in L_s \ | \ \Phi'_{\alpha, s}(w) = I_p - \alpha \nabla^2_w g(w,s) \quad \text{is not invertible} \right\} 
\end{equation*}
where $I_p$ is the identity matrix of size $p$. Diagonalizing $\nabla^2_w g(w,s)$, the determinant of $\Phi'_{\alpha, s}(w)$ is $\prod_{i = 1}^p (1 - \alpha \lambda_i(w, s))$. It is equal to zero if and only if there exists $i \in \left\{1, \ldots, p\right\}$ such that $\alpha \lambda_i(w,s) = 1$ hence $K_{\alpha, s} = \bigcup_{i = 1}^p \left\{w \in L_s \ | \ \alpha\lambda_i(w,s) = 1 \right\}$. 
Since $\alpha \in \Gamma$ and $(\alpha,s) \in \Delta$, by construction of $\Delta$, $\alpha^{-1}$ is a regular value for the functions $w \mapsto \lambda_i(w, s)$, defined for $w \in L_s$, for all $i=1,\ldots,p$. So the set $K_{\alpha,s}$ is a union of $p-1$ dimensional submanifolds in $L_s$ and $K_{\alpha,s}^c$ is open and dense set in $L_s$. Then, let $Z \subset \R^p$ definable and such that $\dim Z \leq p-1$. Assume for the sake of contradiction that there exists a nonempty open set $V \subset \Phi_{\alpha,s}^{-1}(Z)$. The intersection $V \cap K_{\alpha,s}^c$ is open and nonempty because $K_{\alpha, s}^c$ is dense and both sets are open. Since $\Phi_{\alpha,s}$ is a local diffeomorphism on $K_{\alpha,s}^c$, the image $\Phi_{\alpha,s}(V \cap K_{\alpha,s}^c)$ has a nonempty interior but is included in $Z$ of dimension $p-1$,  which is a contradiction. The claim is proved.
\end{proof}

The following claim is a consequence of \Cref{lem:fubiniDefinable}.
\begin{claim} \label{claim:definabilitycorrectset} Let $g \colon \R^p \times \R^m \to \R$ be a definable function and $v \colon \R^p \times \R^m \to \R^p \times \R^m$ be a definable map such that $\nabla_w g = v$ on a definable dense open set $C \subset \R^p \times \R^m$. Then there is a definable set $Z\subset \R^p$, dense, such that for all $w \in Z$, $\nabla_w g(w,s) = v(w,s)$ for all $s$ in a definable dense open set in $\R^m$. 
\end{claim}

\begin{theorem}[Genericity of gradient sequences] \label{prop:backpropisgradientae} 
				Let $g \colon \R^p \times \R^m \rightarrow \R$ and $v : \R^p \times \R^m \rightarrow \R^p$ definable functions. Assume there is a definable dense open set $C \subset \R^p \times \R^m$ such that for all $(w, s) \in C$, $ \nabla_{w} g(w, s) = v(w, s)$. Let $R \subset \R^p$ be the definable dense set such that for all $w \in R$, 
		\begin{itemize}
						\item[--] $v(w,s) = \nabla_w g(w,s)$ for all $s$ in a dense definable set.
						\item[--] $g(\cdot, s)$ is $C^2$ in a neighborhood of $w$, for all $s$ in a dense definable set. 
		\end{itemize}
				Given an arbitrary sequence $(s_k)_{k \in \N}$ in $\R^m$ and an arbitrary $w_0\in \R^p$, consider the recursion
		\begin{equation}
		    \label{eq:sequencebackpropproof}
				w_{k+1} = w_k - \alpha_k v(w_k, s_k) \text{ for all } k \in \N.
		\end{equation}
		Then there is $\Gamma \subset \R$ which complement is finite such that if $\{\alpha_k\}_{k \in \N} \subset \Gamma$, then for each $k \in \N$, there exists a dense definable subset $\Sigma_k \subset \R^p \times (\R^m)^k$ such that if $(w_0,s_0,\ldots,s_{k-1}) \in \Sigma_k$, then $w_i \in R$, for all $i = 0, \ldots,k $.
		In particular there is a full measure residual set, $W \subset \R^p$ such that if $w_0 \in W$  and $(s_0,\ldots, s_{k-1})$ belongs to some definable dense set, then $w_k \in R$.
\end{theorem}

\begin{proof}
				Let $L \subset \R^p \times \R^m$, $\Gamma \subset \R$, $\Delta \subset \Gamma \times \R^m$, given by \Cref{claim:preimageImproved}. $g$ is $C^2$ on $L$ which is definable dense and open, and for every $\alpha \in \Gamma$, the definable set $\{s \in \R^m\ | \ (\alpha,s) \in \Delta\}$ is dense open in $\R^m$.
				
				Recall that $L \cap C \subset \R^p \times \R^m$ is definable dense. By \Cref{lem:fubiniDefinable} there exists a definable dense set $\Sigma_0 \subset \R^p$ such that for all $w \in \Sigma_0$, the set  $\{s \in \R^m|(w,s) \in L \cap C\}$ is dense. It satisfies the desired property, that is $\Sigma_0 \subset R$. Indeed for any $w \in \Sigma_0$, the set $\{s \in \R^m|(w,s) \in L \cap C\}$ is dense and for each such $s$, $(w,s) \in L \cap C$, that is $g$ is $C^2$ at $(w,s)$ and $v(w,s) = \nabla_w g(w,s)$ so that $w \in R$.
				
				We set for all $k$, $\Delta_k = \left\{s \in \R^m \ | \ (\alpha_k, s) \in \Delta \right\}$. 
				By \Cref{claim:preimageImproved}, for any $(\alpha,s) \in \Delta$, $\Phi_{\alpha, s} :=  \operatorname{Id}_p - \alpha \nabla_w g( \cdot, s)$ from $\{w \in \R^p \ | \ (w,s) \in L\}$, dense and open, to $\R^p$ verifies
				\begin{equation}
				\label{eq:dimensionpreimage}
				     \forall Z \subset \R^p \text{ definable},\ \dim Z \leq p-1 \implies \dim \Phi_{\alpha, s}^{-1}(Z) \leq p-1.
				\end{equation}
				Remark that $w_{k+1} = \Phi_{\alpha_k, s_k}\circ\ldots\circ\Phi_{\alpha_0, s_0}(w_0)$, as long as $s_i \in \Delta_i$ for $i= 0,\ldots, k$.

				Let us proceed by induction, fix $k \in \N$ and assume that we have $\Sigma_k \subset R \times \Delta_0 \times \ldots \times \Delta_{k-1}$, definable dense, such that for all $(w_0,s_0,\ldots, s_{k-1}) \in \Sigma_k$, $w_i \in R$ for all $i = 0,\ldots,k$. Note that for $k=0$, $\Sigma_0$ constructed above satisfies the desired hypothesis with the convention that the product set from $0$ to $-1$ is empty.

				Let us construct $\Sigma_{k+1}$. Remark that $s_i \in \Delta_i$ for $i = 0,\ldots, k$ so that $w_{k+1} = \Phi_{\alpha_k, s_k}\circ\ldots\circ\Phi_{\alpha_0, s_0}(w_0)$, as long as $(w_0,s_0,\ldots,s_{k}) \in \Sigma_k \times \Delta_k$.

				Consider the set-valued map $N_{k}\colon s_k \rightrightarrows \Phi_{\alpha_{k}, s_{k}}^{-1}(R)$ and by backward recursion for $i = k-1,\ldots, 0$, $N_{i} \colon (s_i, \ldots, s_k) \rightrightarrows \Phi_{\alpha_{i}, s_{i}}^{-1}(N_{i+1}(s_{i+1}, \ldots, s_k))$. 
				Set 
				$$\Sigma_{k+1} = \{(w,s_0,\ldots,s_k) \in \Sigma_k \times \Delta_k| w \in N_0(s_0,\ldots,s_k)\}.$$
				Let us verify that $\Sigma_{k+1}$ satisfies the desired properties. 
				We have $\Sigma_{k+1} \subset \Sigma_k \times \Delta_k$ so that for any $(w_0,s_0,\ldots, s_k) \in \Sigma_{k+1}$, $(w_0,s_0,\ldots, s_{k-1}) \in \Sigma_k$ and $w_i \in R$ for $i=0,\ldots,k$ by induction hypothesis. 
				Furthermore, $s_i \in \Delta_i$ for all $i = 0,\ldots, k$ so that $w_{k+1} =\Phi_{\alpha_k, s_k}\circ\ldots\circ\Phi_{\alpha_0, s_0}(w_0) $. Note that by construction, $w_0 \in N_0(s_0,\ldots,s_k)$ if and only if $\Phi_{\alpha_k, s_k}\circ\ldots\circ\Phi_{\alpha_0, s_0}(w_0) \in R$ which is the desired property. It remains to show that $\Sigma_{k+1}$ is dense, and the induction will be complete. Note that $\Sigma_{k+1} = \Sigma_k \times \Delta_k \cap \{(w,s_0,\ldots,s_k)|((s_0,\ldots,s_k),w) \in \Graph(N_0)\}$. Since $\Sigma_k \times \Delta_k$ is definable dense and $\Graph(N_0)$ is definable, it suffices to check that $\Graph(N_0)$ is dense. This is done by backward induction.

				Let us first check that $\Graph(N_k)$ is dense.
				We have $(s_k,w_k) \not \in \Graph(N_k)$ if and only if $(w_k,s_k) \not \in L$ ($\Phi_{\alpha_k, s_k}$ is not defined at $w_k$), or $w_k  \in \Phi_{\alpha_k, s_k}^{-1}(R^c)$ so that $\Graph(N_k)^c = L^c \cup \{(s_k,w_k)| (w_k,s_k) \in L,\, w_k \in \Phi_{\alpha_k, s_k}^{-1}(R^c)\}$. Recall that $R^c$ is definable and is the complement of a dense set, therefore it has at most dimension $p-1$ so that if $s_k \in \Delta_k$, $\Phi_{\alpha_k, s_k}^{-1}(R^c)$ also has dimension at most $p-1$. On the other hand, the set $\left\{w_k \in \R^p| (w_k, s_k) \in L^c \right\}$ is the complement of $L_{s_k}$ (with the notation of \Cref{claim:preimageImproved}) definable and dense for $s_k \in \Delta_k$. Therefore for all $s_k \in \Delta_k$, the set $\{w_k \in \R^p|(s_k,w_k) \not \in \Graph(N_k)\}$ has dimension at most $p-1$ and the set $\{w_k \in \R^p|(s_k,w_k) \in \Graph(N_k)\}$ is dense so that $\Graph(N_k)$ is dense by Lemma \ref{lem:fubiniDefinable}.
				
				This extends by backward induction. Assume that $\Graph(N_{i+1})$ is dense for some $i \in \{0,\ldots, k-1\}$. We have $(s_i,\ldots, s_k, w_i) \not\in \Graph(N_i)$ if and only if $(w_i,s_i) \not \in L$ ($\Phi_{\alpha_i, s_i}$ is not defined at $w_i$) or $w_i \in \Phi^{-1}_{\alpha_i,s_i}\left(N_{i+1}(s_{i+1},\ldots,s_{k})^c\right)$. As $N_{i+1}$ has a dense graph, then by \Cref{lem:fubiniDefinable}, for all $(s_{i+1}, \ldots, s_k)$ in a dense definable set $R_i$, $N_{i+1}(s_{i+1},\ldots,s_{k})$ is dense and $N_{i+1}(s_{i+1},\ldots,s_{k})^c$ has dimension at most $p-1$. Therefore, similarly as above, for all $s_i \in \Delta_i$ and $(s_{i+1}, \ldots, s_k) \in R_i$, the set $\{w_i|(s_i,\ldots,s_k,w_i) \in \Graph(N_i)\}$ is dense and $\Graph(N_{i})$ is dense. By induction, $\Graph(N_0)$ is dense and this shows that $\Sigma_{k+1}$ has the correct property.

				This proves the first statement. Now by \Cref{lem:fubiniDefinable}, for each $k \in \N$, there is $W_k \subset \R^p$ definable dense such that for each $w_0 \in W_k$, for all $(s_0,\ldots, s_{k-1})$ in a dense definable set, $w_i \in R$ for all $i = 0, \ldots, k$. We set $W = \cap_{k \in \N} W_k$, $W$ is a residual set by countable intersection of residual sets (with dense interior), and it has full measure as a countable intersection of full measure sets. 
\end{proof}

\begin{remark}
				With the notation of \Cref{prop:backpropisgradientae}, if $(s_i)_{i\in\N}$ are independent and identically distributed with a density with respect to Lebesgue measure, $w_0 \in W$ and $\alpha_k \in \Gamma$ for all $k \in \N$, then almost surely, $w_k \in R$ for all $k \in \N$.
				\label{rem:avoidArtifactProba1}
\end{remark}

\section{Proofs of Section 2}
\label{section:proofresults}
\begin{proof}[Proof of \Cref{corr:pathdiffexpectation}] Let us show that \Cref{theorem:leibnizconservative} applies. 1: measurability of $f$ follows from semialgebraicity in \Cref{ass:convergence_essacc} (3) and integrability follows from \Cref{ass:convergence_essacc} (1). 2: these also entail that for all $s \in \R^m$, $f(\cdot,s)$ is locally Lipschitz and path-differentiable \cite[Proposition 2]{bolte2019conservative} with conservative gradient $\partial^c_w f(\cdot,s)$. 3: by semialgebraicity of $f$, the set-valued map $\partial^c_w f$ is semialgebraic, hence jointly measurable. 
Toward 4, let $K \subset \R^p$ compact and $M > 0$ such that $K \subset B(0, M)$. Let $\kappa$ be given by \Cref{ass:convergence_essacc} (1), we have $\|\partial^c_w f(x, s)\| \leq \kappa(s) (1 + M^{q_0})$ for all $x \in K$, for almost all $s \in \R^m$ which implies 4 because $\kappa$ is integrable. 

By \Cref{theorem:leibnizconservative}, $\E_{\xi \sim P}[\partial^c_w f(\cdot, \xi)]$ is a conservative gradient for $\mathcal{J}$ and $\mathcal{J}$ is path-differentiable and in particular admits a chain rule with respect to $\partial^c \mathcal{J}$.
\end{proof}

\begin{proof}[Proof of \Cref{th:convergence_essacc}]
			
				We first show \ref{th:convergence_essac_1} which is an application of the first part of \Cref{theorem:subsequentialconvergence}. We need to verify that \Cref{ass:convergence_essacc} is sufficient for \Cref{ass:algorithmicAssumptionStochApproximation} to hold. We let $D_\mathcal{J}(\cdot) = \E_{\xi \sim P}[\partial^c_w f(\cdot,\xi)]$ which is convex valued because integration in \Cref{def:aumann} preserves convexity. Furthermore, for all $s \in \R^m$, $\partial^c_w f(\cdot,s)$ is a conservative gradient for $f(\cdot,s)$, because $f$ is semialgebraic, hence path differentiable. 
				We are therefore in the setting of \Cref{ass:algorithmicAssumptionStochApproximation}. 1: measurability of $f$ follows from semialgebraicity in \Cref{ass:convergence_essacc} (3) and integrability follows from \Cref{ass:convergence_essacc} (1). 2: by semialgebraicity of $f$, $(w,s) \mapsto \partial^c_w f(w,s)$ is also semialgebraic and therefore jointly measurable. \Cref{ass:convergence_essacc} implies that for all $s$ and all $w$, $f(\cdot,s)$ is $\kappa(s)(1+\|w\|^{q_0})$ Lipschitz for on $B(0,\|w\|)$, so that $\partial^c_w f(w,s)$ satisfies the same bound which is \Cref{ass:integralconservative}.3. The rest of \Cref{ass:algorithmicAssumptionStochApproximation} follows from \Cref{ass:convergence_essacc} and \Cref{th:convergence_essac_1} follows from \Cref{theorem:subsequentialconvergence}. 

				We now turn to \cref{th:convergence_essac_2}. The arguments are the same, except that a smaller set-valued field drives the dynamics thanks to \Cref{prop:backpropisgradientae} and an interchanging of integral and derivative. As above, the function $f$ is semialgebraic. Let $C := \left\{(w,s) \in \R^p \times \R^m \ | \ \partial^c_w f(w,s) = \{\nabla_w f(w,s)\} \right\}$. By \Cref{ass:convergence_essacc}, for almost all $s \in \R^m$, $\partial^c_w f(\cdot,s)$ is a conservative gradient for $f(\cdot,s)$, so that $\{w \in \R^p|(w,s) \in C\}$ is a full measure definable set \cite[Theorem 1]{bolte2019conservative} hence dense definable. By \Cref{lem:fubiniDefinable}, $C$ is also dense. Since for all $(w,s)$, $v(w,s) \in \partial^c f(w,s)$, for all  $(w,s) \in C \subset \R^p \times \R^m$, $\nabla_w f(w,s) = v(w,s)$ and we can apply \Cref{prop:backpropisgradientae} and \Cref{rem:avoidArtifactProba1}. There is a full measure residual set $W \subset \R^p$, and $\Gamma \subset \R$ whose complement is finite such that if $w_0 \in W$ and $\alpha_k \in \Gamma$ for all $k \in \N$, almost surely, for all $k \in \N$, we have $v(w_k, s_k) = \nabla_w f(w_k,s_k)$ and $w_k \in R$, where $R \subset \R^p$ is a dense definable set such that for all $w \in R$:
		\begin{itemize}
						\item[--] $v(w,s) = \nabla_w g(w,s)$ for all $s$ in a dense definable set.
						\item[--] $f(\cdot, s)$ is $C^2$ in a neighborhood of $w$, for all $s$ in a dense definable set. 
		\end{itemize}
		
				Fix $w \in R$ and let $\kappa$ be as in \Cref{ass:convergence_essacc}. For all $a \in [-1, 1]$ and $h \in \R^p$, $ |f(w + ah, s) - f(w, s)| \leq \kappa(s)(1 + (\|w\|+ \|h\|)^{q_0} ) \|h\|$ for $P$-almost all $s \in \R^m$. The right-hand side is integrable, and we may apply the dominated convergence theorem,
				\begin{align*}
				    &\lim_{a \to 0}   \frac{ \mathcal{J}(w + ah) - \mathcal{J}(w)}{a} \quad=\quad\lim_{a \to 0} \frac{1}{a} \int_{\R^m} f(w + ah, s) - f(w, s) \di P(s)\\ 
				    =\;& \int_{\R^m}  \lim_{a \to 0} \frac{f(w + ah, s) - f(w, s)}{a}\di P(s) \quad=\quad \int_{\R^m} \langle \nabla_w f(w, s), h \rangle \di P(s) \\ 
					=\;& \left\langle \int_{\R^m}  \nabla_w  f(w, s)\di P(s), h \right\rangle,
				\end{align*}
				where the limit under the integral is because for all $s$ in a dense definable set (hence almost all $s$), $f(\cdot, s)$ is $C^2$ in a neighborhood of $w$. This shows, in particular, that $\mathcal{J}$ is differentiable at $w$.

				We set $\tilde{D}_\mathcal{J} = \partial^c \mathcal{J}$, $\bar{v}$ a measurable selection in $\tilde{D}_\mathcal{J}$ such that $v(w) = \nabla \mathcal{J}(w)$ whenever $\mathcal{J}$ is differentiable at $w$. This means that for all $w$, there is $v_w \colon \R^m \to \R^p$ integrable such that $\int v_w(s) dP(s) = \bar{v}(w)$ and for all $s$, $v_w(s) = \nabla \mathcal{J}(w)\in \partial_w^c f(w,s)$. Set $\tilde{v}(w,s) = v_w(s)$ for all $(w,s)$ and let $L$ be as in \Cref{claim:preimageImproved} definable, open and dense. We have that for all $(w,s) \in L$, $g(\cdot,s)$ is $C^2$ in a neighborhood of $w$ so that $\tilde{v}(w,s) = \nabla_w f(w,s)$. Since $\tilde{v}$ agrees with $\nabla_w f(w,s)$ (definable, measurable) on the full measure set $L$ (definable, dense), it is therefore jointly measurable \cite[Proposition 3, Section 18.1]{royden1968real} and satisfy for all $w$, $\int \tilde{v}(w,s)dP(s) \in \tilde{D}_\mathcal{J}(w)$.

				Now, since for all $k \in \N$, $w_k \in R$ almost surely, it holds that almost surely, 
				$\mathcal{J}$ is differentiable at $w_k$ with $\nabla \mathcal{J}(w_k) = \E \left[\nabla_w f (w_k, \xi_k)|w_k \right] = \E \left[v (w_k, \xi_k)|w_k \right]$ and $v(w_k,s_k) = \nabla_w f(w_k,s_k) = \tilde{v}(w_k,s_k)$. We actually have for all $k \in \N$
				\begin{align*}
								&w_{k+1} = w_k - \alpha_k \tilde{v}(w_k,s_k),\\
								&\int_{\R^m} \tilde{v}(w_k,s) \di P(s) = \nabla \mathcal{J}(w_k), \text{ and}&
								\int_{\R^m} \tilde{v}(w,s) \di P(s) &\in \partial^c \mathcal{J}(w),\, \forall w \in \R^p
				\end{align*}
				which is equivalent to the recursion given in \Cref{th:convergence_essac_2}. To obtain the desired convergence result we may apply \Cref{theorem:subsequentialconvergence} similarly as above, with $\tilde{D}_\mathcal{J} = \partial^c \mathcal{J}$ in place of $D_\mathcal{J}$ and $\tilde{v}$ in place of $v$ and the result follows.
\end{proof}

\begin{proof}[Proof of \Cref{th:algosemialgebraic}]
				We have already shown that \Cref{ass:convergence_essacc} is sufficient to ensure that \Cref{ass:algorithmicAssumptionStochApproximation} holds so that \Cref{theorem:subsequentialconvergence} applies. The claimed statement is indeed the second statement of \Cref{theorem:subsequentialconvergence} which holds \Cref{ass:summabilitystepsizes}. So we simply need to check that \Cref{ass:morsesard} holds.

				As for the first statement, under \Cref{ass:convergence_acc}, the function $\mathcal{J}$ is definable using \cite[Theorem 1.3]{Cluckers_2011}. From \Cref{proposition:definabilityofsetvaluedintegral} the set-valued map $D_{\mathcal{J}} = \E_{\xi \sim P} \left[\partial^c f(\cdot, \xi) \right]$ is definable and it is a conservative gradient for $\mathcal{J}$ by \Cref{corr:pathdiffexpectation}. By definable conservative Sard's theorem \cite{bolte2019conservative}, the set of $D_{\mathcal{J}}$-critical values of $\mathcal{J}$ which is $\mathcal{J}(\crit D_{\mathcal{J}})$ is finite, hence \Cref{ass:morsesard} is satisfied. The result is then a consequence of the second statement of \Cref{theorem:subsequentialconvergence} under hypothesis 2 of \Cref{lem:summabilityNoise}.

				Regarding the second statement, \Cref{th:convergence_essacc} gives us $\Gamma \subset \R$ whose complement is finite, $W \subset \R^p$ of full measure and residual, such that if $\alpha_k \in \Gamma$ for all $k \in \N$ then for all initialization $w_0$ in $W$ we have with probability $1$ the recursion holds with $\tilde{v} = \nabla_w f$ in place of $v$ and $\tilde{D}_\mathcal{J} := \partial^c \mathcal{J}$ in place of $D_\mathcal{J}$. Hence, almost surely, one has an actual stochastic subgradient sequence so that the result follows as above from \Cref{theorem:subsequentialconvergence} applied to $\tilde{v}$ and $\tilde{D}_\mathcal{J}$ which are both semialgebraic. 
\end{proof}

\begin{proof}[Proof of {\rm\Cref{corr:convergenceclarke}}]
We prove the regression case, i.e., when $P$ has a compactly supported semialgebraic density $\phi$ with respect to the Lebesgue measure. The classification case uses similar arguments, see \Cref{remark:discretecontinuousdistribution} (2) We apply \Cref{th:convergence_essacc}, which holds replacing the Clarke subgradient by a general conservative gradient $D$, see \Cref{remark:conservativeinsteadofclarke} (2) and \Cref{theorem:subsequentialconvergence} similarly as above. We verify that in the setting of \cref{subsection:deeplearning}, Assumptions \ref{ass:convergence_essacc}, \ref{ass:convergence_acc} and hypothesis 3 of \Cref{lem:summabilityNoise} are satisfied.
Under \Cref{ass:lipschitzandsemialgebraic}, let $D : \R^p \times \R^d \times \R^I \rightrightarrows \R^p$ be the product of Clarke Jacobians in (\ref{eq:defBackprop}), it is semialgebraic, for almost all $s \in \R^d \times \R^I$, $\backprop_wf(\cdot, s)$ is a semialgebraic selection of $D(\cdot, s)$, and $D(\cdot, s)$ is a conservative gradient for $f(\cdot, s)$.  As $D$ is semialgebraic, it is polynomially bounded hence there exist $K > 0$, $p_0  \in \N$, $q_0 \in \N$ such that for all $w \in \R^p$ and almost all  $s \in \R^d \times \R^I$, $\|D(w,s)\| \leq K(1 + \|s\|^{p_0}) (1 + \|w \|^{q_0})$.  Let $\kappa(s) := K(1+ \|s\|^{p_0})$ for $s \in \R^d \times \R^I$. Since $P$ has compact support, then the function $\kappa$ is square integrable with respect to $P$. \Cref{ass:convergence_essacc} (1) is satisfied. $f$ and $\backprop_w$ are semialgebraic by assumption, hence \Cref{ass:convergence_essacc} (3) is satisfied. The other assumptions are directly satisfied by assumptions of \Cref{corr:convergenceclarke}. Note that for any $w$, $\sup_{s \in \operatorname{supp} P} \|\backprop_wf(w,s) \|\leq \tilde{K} (1 + \|w\|^{q_0})$ which is locally bounded so that we may apply \Cref{theorem:subsequentialconvergence} under the hypothesis 3. of \Cref{lem:summabilityNoise} similarly as in the proof of \Cref{th:algosemialgebraic}.

\end{proof}

\section{Generalized gradients of Norkin and conservativity}
\label{section:norkin}
\subsection{Definitions}
Throughout this section, $f \colon \R^p \to \R$ is Lipschitz continuous and $D \colon \R^p \rightrightarrows \R^p$ is locally bounded nonempty convex valued and upper semicontinuous. Convex values are indeed required by Norkin in \cite{norkin1978nonlocal,norkin1980,norkin1986}.
\begin{definition}[Semismooth generalized gradients]{\rm
	The set-valued mapping 	$D$ is a {\it generalized gradient of $f$}  if for all $x \in \R^p$, we have
		\begin{align*}
				\underset{y \to x,\, g \in D(y)}{\lim\sup} \frac{f(y) - f(x) - \left\langle g, y-x\right\rangle}{\|y-x\|} = 0.
		\end{align*}
		\label{def:norkinGradient}}
\end{definition}
The $\lim\sup$ property in the definition is referred to as the {\em semismoothness property} of the generalized gradients. 
On the other hand, conservative gradients are defined in \Cref{def:conservativeField}. In both cases, the corresponding set-valued gradient map is a singleton almost everywhere and contains the Clarke subgradient of $f$  everywhere \cite{norkin1978nonlocal,bolte2019conservative}. Functions with generalized gradient are called  {\em differentiable in the generalized sense}, and those with conservative gradients are called path-differentiable.
\subsection{Relations between the two notions}
The following strengthens the chain rule along semismooth curves given in \cite[Theorem 1]{Ruszczynski2020}.
\begin{proposition}
		If $D$ is a generalized gradient of $f$ in the sense of Definition \ref{def:norkinGradient}, then it is a conservative gradient of $f$.
\end{proposition}
\begin{proof}
	We shall use the chain rule characterization of conservative gradients, \Cref{def:conservativeField}.
	Let $D$ be a generalized gradient of $f$  as in \Cref{def:norkinGradient}, and $\gamma \colon[0,1] \to \R^p$ be an absolutely continuous path. Then both $\gamma$ and $f\circ \gamma$ are absolutely continuous, hence differentiable almost everywhere. 	Therefore, there exists a full measure subset $R \subset [0,1]$ such that both are differentiable at every point on $R$. 
	
	Suppose, toward a contradiction, that the chain rule is not valid along $\gamma$, that is, there exists a non zero measure set $E_1 \subset R$ such that for all $t \in E_1$, there is $g \in D(\gamma(t))$ such that $\frac{\di}{\di t} (f\circ\gamma)(t) \neq \left\langle g, \dot{\gamma}(t)\right\rangle$. Note that this implies that $\dot{\gamma}(t) \neq 0$ for all $t \in E_1$, since if $\dot{\gamma}(t)= 0$ then $0 = \frac{\di}{\di t} (f\circ\gamma)(t) = \left\langle g, \dot{\gamma}(t)\right\rangle$. Reducing $E_1$ and changing sign if necessary, we may assume without loss of generality that for all $t \in E_1$, there is $g \in D(\gamma(t))$ such that $\frac{\di}{\di t} (f\circ\gamma) (t) < \left\langle g, \dot{\gamma}(t)\right\rangle$.
	
	Consider the measurable function (measurability is justified in \cite{bolte2019conservative}), $g \colon [0,1] \to \R^p$, defined for all $t \in R$ by $g(t) = \arg\max_{v \in D(\gamma(t))} \left\langle \dot{\gamma}(t), v\right\rangle$ and $g(t) = 0$ otherwise. 

	We have for all $t \in E_1$, $0 < \left\langle \dot{\gamma}(t), g(t)\right\rangle - \frac{\di}{\di t} (f\circ\gamma) (t)$. This means that there is $\epsilon > 0$ and a nonzero set $E_2 \subset E_1$ such that $\epsilon \leq \left\langle \dot{\gamma}(t), g(t)\right\rangle - \frac{\di}{\di t} f\circ\gamma (t)$ for all $t \in E_2$ (otherwise, one would have $ \left\langle \dot{\gamma}, g\right\rangle - \frac{\di}{\di t} (f\circ\gamma) = 0$ almost everywhere on $E_1$).
	
	Let us apply Lusin's theorem (see, e.g.,  \cite[Section 3.3]{royden1968real}) and fix an arbitrary $\alpha>0$, such that $\lambda(E_2)>\alpha$. There is a closed subset $E_3 \subset E_2$ such that $\lambda(E_2 \setminus E_3) < \alpha$ and $g$ restricted to $E_3$ is continuous. The set $E_3$ has positive measure since $\lambda(E_3) = \lambda(E_2) - \lambda(E_2 \setminus E_3)> \alpha - \alpha = 0$. Let us summarize, $E_3 \subset [0,1]$ is closed with positive measure and we have the following on $E_3$:
	\begin{itemize}
			\item[--] Both $f\circ \gamma$ and $\gamma$ have derivatives and $\dot{\gamma} \neq 0$.
			\item[--] $\frac{\di}{\di t}( f\circ\gamma)  + \epsilon \leq \left\langle \dot{\gamma}, g \right\rangle$.
			\item[--] $g$ restricted to $E_3$ is continuous.
	\end{itemize}
Lebesgue density theorem (see, e.g., \cite[Theorem 1.35]{evans1992measure}) ensures that almost all $t \in E_3$ have density $1$, that is,
$ \lambda([t-\delta,t + \delta] \cap E_3) / \lambda([t-\delta,t + \delta]) \to 1$ as $\delta \to 0$.
Since $E_3$ has positive measure, there exists $\bar{t} \in E_3$, a point of density $1$ in $E_3$. We have for all $t \neq \bar{t}$, such that $\gamma(t) \neq \gamma(\bar{t})$, 
\begin{align*}
	 &\frac{f(\gamma(t)) - f(\gamma(\bar{t}))}{(t-\bar{t})} = \frac{\|\gamma(t) - \gamma(\bar{t})\|}{t-\bar{t}}\frac{f(\gamma(t)) - f(\gamma(\bar{t}))}{\|\gamma(t) - \gamma(\bar{t})\|} \\
	 =\;& \frac{\|\gamma(t) - \gamma(\bar{t})\|}{t-\bar{t}}\ \left(\frac{f(\gamma(t)) - f(\gamma(\bar{t})) - \left\langle g(t), \gamma(t) - \gamma(\bar{t})\right\rangle}{\|\gamma(t) - \gamma(\bar{t})\|} + \frac{\left\langle g(t), \gamma(t) - \gamma(\bar{t})\right\rangle}{\|\gamma(t) - \gamma(\bar{t})\|}\right)\\
	 	=\;& \frac{\|\gamma(t) - \gamma(\bar{t})\|}{(t-\bar{t})}\ \left(\frac{f(\gamma(t)) - f(\gamma(\bar{t})) - \left\langle g(t), \gamma(t) - \gamma(\bar{t})\right\rangle}{\|\gamma(t) - \gamma(\bar{t})\|}  \right) + \frac{\left\langle g(t), \gamma(t) - \gamma(\bar{t})\right\rangle}{t -\bar{t}}.
\end{align*}
Letting $t \to \bar{t}$ with $t \in E_3$, $t \neq \bar{t}$ and $\gamma(t) \neq \gamma(\bar{t})$, which is possible because $\bar{t}$ has density $1$ in $E_3$ and $\dot{\gamma}(\bar{t}) \neq 0$, we have
		\begin{align*}
			 &\frac{f(\gamma(t)) - f(\gamma(\bar{t}))}{(t-\bar{t})} \to \frac{\di}{\di t}( f\circ \gamma)(\bar{t}),
			 &\frac{\|\gamma(t) - \gamma(\bar{t})\|}{(t-\bar{t})} \to \|\dot{\gamma}(\bar{t})\| \\
			 &\frac{f(\gamma(t)) - f(\gamma(\bar{t})) - \left\langle g(t), \gamma(t) - \gamma(\bar{t})\right\rangle}{\|\gamma(t) - \gamma(\bar{t})\|}  \to 0,
			 &\frac{\left\langle g(t), \gamma(t) - \gamma(\bar{t})\right\rangle}{t -\bar{t}} \to \left\langle g(\bar{t}), \dot{\gamma}(\bar{t})\right\rangle,
	\end{align*}
	where the two identities on the first line follow from the differentiability of $f\circ \gamma$ and $\gamma$ at $\bar{t} \in E_3$, the third one stems from the semismooth property of generalized gradients  (\Cref{def:norkinGradient}) while the last one is by differentiability of $\gamma$ and continuity of $g$ restricted to $E_3$ at $\bar{t}$.
	We obtain that  $  \frac{\di}{\di t} (f\circ \gamma)(\bar{t}) = \left\langle g(\bar{t}), \dot{\gamma}(\bar{t})\right\rangle \geq \frac{\di}{\di t} (f\circ \gamma)(\bar{t}) + \epsilon$, where the equality follows by the previous limit and the inequality is because $\bar{t} \in E_3$. This is contradictory since $\epsilon > 0$, which concludes the proof.	
\end{proof}

Functions differentiable in the generalized sense are path-differentiable. In the semialgebraic case, both notions coincide \cite{davis2021conservative}, but the inclusion is strict in general.
\begin{proposition}
		Consider the closed set $C \subset [-1,1]$ defined through $C = \{1/k \ | \  k \in \Z, k\neq 0\}\cup \{0\}$.  Then the distance function $F$ to $C$ is path-differentiable but not differentiable in the generalized sense.
\end{proposition}
\begin{proof}
	First let us recall a substitution formula for absolutely continuous function \cite[Corollary 7]{serrin1969general}. If $g \colon \R \to \R$ is absolutely continuous and $f \colon \R \to \R$ is measurable and bounded, then for all $\alpha,\beta$, $\int_{g(\alpha)}^{g(\beta)} f(x) dx = \int_\alpha^\beta f(g(s)) \dot{g}(s) ds$.
	
	It is clear that $\partial^c F$ is locally constant ($+1$ or $-1$) out of a closed countable set (the set $C$ and its cut locus). Therefore, choosing $f$ to be any measurable selection in $\partial^c F$, the previous formula allows concluding that $F$ is path-differentiable. Indeed $F$ is path-differentiable if and only it satisfies the change of variable formula for any absolutely continuous $g$, which is the case because $F$ is $1$-Lipschitz so that $|f| \leq 1$.
	
	On the other hand,  $F(0) = 0$ and for all $k \in \N^*$, $F(1/k) = 0$ and $\partial^c F(1/k) = [-1,1]$ so that $-1 \in \partial^c F(1/k)$. The equality $\frac{F(1/k) - F(0) - \left\langle -1, 1/k- 0\right\rangle}{\|1/k-0\|} = 1$
	contradicts the semismoothness property for the Clarke subgradient of $F$. Since $F$ is differentiable in the generalized sense if and only if its Clarke subgradient is a  generalized gradient, we conclude that $F$ is not differentiable in the generalized sense at $0$.
\end{proof}
\section*{Acknowledgements}
This work was funded by the AI Interdisciplinary Institute ANITI funding under the grant agreement ANR-19-PI3A-0004, Air Force Office of Scientific Research, Air Force Material Command, USAF, under grant numbers FA9550-19-1-7026,  and ANR MaSDOL 19-CE23-0017-01.  J. Bolte  also acknowledges the support of ANR Chess, grant ANR-17-EURE-0010 and TSE-P.

\bibliographystyle{plain}
\bibliography{references}
\end{document}